\documentclass{amsart}
\usepackage{times}
\usepackage{amsthm}
\usepackage{amssymb}
\usepackage[pdftex]{graphicx}
\usepackage[all]{xy}
\usepackage[mathscr]{eucal}
\usepackage{amsmath,latexsym,oldlfont}
\usepackage{mathdots}
\usepackage{pifont}
\usepackage{stmaryrd}
\usepackage{textcomp}
\usepackage{pifont}
\usepackage{multirow}
\usepackage[colorlinks,linkcolor=blue,citecolor=red,linktocpage=true]{hyperref}
\usepackage{mathtools}
\numberwithin{equation}{section}
\newtheorem{teorema}{Teorema}[section]
\newtheorem{proposition}[teorema]{Proposition}
\newtheorem{theorem}[teorema]{Theorem}
\newtheorem{lemma}[teorema]{Lemma}
\newtheorem{ejemplo}[teorema]{Ejemplo}

\newtheorem{corollary}[teorema]{Corollary}
\newtheorem{definition}[teorema]{Definition}

\newcommand{\swedge}{{\scriptstyle\wedge}}

\newcommand{\nwedge}{\mathchoice{{\textstyle\wedge}}%
    {{\wedge}}%
    {{\textstyle\wedge}}%
    {{\scriptstyle\wedge}}}
\newsavebox{\spacebox}
\begin{lrbox}{\spacebox}
    \verb*! !
\end{lrbox}

\title[ On connection Bruijn-Erd\"os Theorem and symplectic geometry]{ On connection Bruijn-Erd\"os Theorem and symplectic geometry}
\author[J. Carrillo--Pacheco]{Jes\'us Carrillo--Pacheco}

\address[J. Carrillo--Pacheco ]{Academia de Matem\'aticas, Universidad Aut\'onoma de la Ciudad de M\'exico, 09390 Ciudad de M\'exico,  M\'exico.}
\email[J. Carrillo--Pacheco]{jesus.carrillo@uacm.edu.mx}
 
\begin{document}



\begin{abstract}
{In this article a family of recursive and self-similar matrices is constructed. It is shown that the Pl\"ucker matrix of the Isotropic Grassmannian variety is a direct sum of this class of matrices.}
\end{abstract}
 \maketitle
\section{Introduction}\label{Secc0}
    The de  Bruijn - Erd$\ddot{o}$s theorem  \cite{bib 0.2},  was proved by Nicolaas Govert de Bruijn and Paul Erd$\ddot{o}s$ $(1951)$.    
In its classical form establishes: 
Let ${\EuScript S}$ be a finite set consisting of $n$ elements, and let $L=\{ S_1, \ldots,  S_m \}$ be a collection of $m$ proper subsets of ${\EuScript S}$. Suppose that\\
 $I)$ each $S_i$ contains at least two elements, \\
$II)$  for every $x, \; y \in{\EuScript S}$, $x\neq y$  there exists exactly one subset $S_i$ such that $x, y \in S_i$\\
   Then $m\geq n$, and $m = n$ if and only if one of the following two cases occurs:\\
$(i)$ there exists $z \in {\EuScript S}$  such that $L$ consists of ${\EuScript S}-\{z \}$ and all the pairs $\{x, z\}$ with $x\neq z$;
$(ii)$ there exists a natural number $k$ such that $n = k(k-1) + 1$ and\\
$a)$ $|S_i|=k$ for every  $i=1,\ldots, m$\\
$b)$ $|S_i\cap S_j|\geq 1$ for al $i\neq j$ \\ 
$c)$ every element of ${\EuScript S}$ belongs to exactly $k$ subsets $S_i$ \\
see  \cite{bib 2.32}. This theorem is very important in discrete mathematics see \cite{bib 2.231} for some applications.\\
  Ryser study an  Bruijn - Erd$\ddot{o}s$ type theorems that deal with the foundations of finite geometries see  \cite[Theorem 1.1]{bib 5b}. 
  
%

\subsection{A bizarre version of the Bruijn-Erdos theorem}
For each pair of integers $2\leq k \leq n$  there exists a set $S=\{s_1, \ldots, s_N \}$ and a family of proper subsets  $\{S_1, \ldots, S_M \}\subseteq S$, such that $M< N$ and 
\begin{description}
\item[I)] $|S_i|=n-\frac{k-2}{2}$ for all $i=1,\ldots, M$
\item [II)]  $|S_{P_{\alpha}}\cap S_{P_{\overline{\alpha}}} |\leq 1$ for all $i\neq j$ and $i, j=1,\ldots, M$
\item[III)] Each $S_i$ has nonempty interseccion  with exactly $k/2$ of the other subsets. 
\item[IV)] The incidence matrix of the family of subsets is a $(0, 1)$-matrix that has a fixed number of
ones in each row and a fixed number of ones in each column, it is also a sparse and self-similar matrix.
\item[V)] There is a matrix, direct sum of matrices of type $IV$ above, which allows determining the Isotropic grassmannian variety, as a subfamily of solutions to the homogeneous system generated by said matrix.  
\end{description}
This article is a continuation of \cite{bib 2.20}, where a study of the Lagrangian-Grassmannian variety is made, which is a particular case of the Isotropic-Grassmannian variety $IG(k, E)$. Here in the present article, the Pl\"ucker matrix of $IG(k, E)$ is exclusively studied, we determine its structure as a direct sum of a family of submatrices, which we call fractal matrices because they are built with a recursive algorithm and have properties of self-similarity.
This article is organized as follows: in section 3 the pl\"ucker matrix is defined, in section 4 it is shown that the Pl\"ucker matrix is a direct sum of a family of submatrices, in section 5 it is they define the fractal matrices and it is shown that the pl\"ucker matrix is a direct sum of fractal matrices. Section 6 gives a method to find the rational points when we consider a finite field.
\section{Preliminary}\label{Secc1}
\subsubsection{indices} Let $m$ be an integer we denote by
  \begin{equation}\label{mbrak}
   [m]:=\{1, 2, \ldots, m \}
   \end{equation}
    to the set of the first $m$ integers.
Let $m$ and $\ell$ be positive integers such that  $\ell< m$ as usual in the literature 
$C^m_{\ell}$ denotes binomial coefficient. If  $\alpha=(\alpha_1,\ldots,\alpha_{\ell})\in {\mathbb N}^{\ell}$ then we define $supp\{ \alpha\}:=\{ \alpha_1, \cdots, \alpha_{\ell}\}$. 
 If $s\geq 1$ is a positive integer and  $\Sigma\subset {\mathbb N}$ is a non-empty set,   
we define  the  sets
\begin{equation}\label{setcomb1}
C_s(\Sigma):=\{\alpha=(\alpha_1, \ldots, \alpha_{s})\in {\mathbb N}^s: \alpha_1< \ldots < \alpha_{s} \;and\; supp \{\alpha \}\subseteq \Sigma \}
\end{equation}
Clearly  if $|\Sigma|=m$ then $|C_s(\Sigma)|=C^m_{\ell}$, 
whit this notation if $\ell < m$  we define $I(\ell, m):=C_{\ell}([m])$. so we have
\begin{equation}\label{setcomb2}
I(\ell,m)=\{\alpha=(\alpha_1, \ldots, \alpha_{s}) :
1\leq\alpha_1<\cdots<\alpha_{\ell}\leq m \}
\end{equation}
In general we say 
\begin{equation}
\alpha=(\alpha_1, \ldots, \alpha_s)\in C_s(\Sigma)  
\end{equation}
it there is a permutation $\sigma$ such that arrange the elements of $supp \{\alpha\}$ in increasing order
then we agreed $(\sigma(\alpha_1), \ldots, \sigma(\alpha_s))\in C_s(\Sigma)$. \\
Also if $\alpha$ and $\beta$ are elements of $C_s(\Sigma)$ then we say that 
\begin{equation}\label{suppiguls}
\alpha=\beta\;  \Leftrightarrow \;  {\text supp}\{\alpha\}={\text supp}\{\beta\}
\end{equation} 
Let $\alpha \in I(n,2n)$, suppose there are $i, j \in supp\; \alpha$ such that $i+j=2n+1$ in this case 
$j=2n-i+1$ and we write this pair as $P_i= (i, 2n-i+1)$ so we define the set  
\begin{equation}\label{setSigma0001}
\Sigma_n=\big\{ P_1, \ldots, P_n\big\}
\end{equation}
 and if  $\alpha \in I(n,2n)$ such that $\{ i, 2n-i+1\}\subset supp\{ \alpha \}$, then we say that 
$P_i\in supp\{ \alpha \}$ and that $P_i\in supp\{\alpha\}\cap \Sigma_m$.
We denote 
\begin{equation}\label{sigmaaaa1}
C_{\frac{k}{2}}(\Sigma_n):=\{P_{\beta}=(P_{\beta_1},\ldots, P_{\beta_{k/2}}):\beta\in I(k/2, n) \}
\end{equation}
and 
\begin{equation}\label{sigmaaaa2}
C_{\frac{k-2}{2}}(\Sigma_n):=\{P_{\alpha}=(P_{\alpha_1},\ldots, P_{\alpha_{(k-2)/2}}):\alpha\in I((k-2)/2, n) \}
\end{equation}

If $1\leq a_1<a_2<\cdots<a_{2\ell}\leq 2n$ such that $a_i+a_j\neq 2n+1$  then we define 
\begin{equation}\label{setSigma0011}
\Sigma_{a_1,\ldots,a_{2\ell}}=\Sigma_n-\{P_{a_1}, \dots, P_{a_{2\ell}}\}
\end{equation}
so $|\Sigma_{a_1,\ldots,a_{2\ell}}|=n-2\ell$. 
Clearly if $n\geq 4$ then for each $\alpha\in I(n, 2n)$ we have to $|supp \{\alpha\}\cap \Sigma_n|=0$ or $|supp \{\alpha\} \cap \Sigma_n|=|supp \{\alpha\}|$ or $1\leq |supp \{\alpha\}\cap \Sigma_n|\leq\lfloor\frac{n-2}{2}\rfloor$ respectively.
Note that $|supp P_{\beta}|=2|supp \beta|$.\\
\subsection{Incidence}\label{prelimin1}
Following \cite[pag. 3]{bib 0.1} we call  $X=\{ x_1, \ldots, x_n\}$ an $n$-set. Now let $X_1, X_2, \dots, X_m$ be $m$ not necessarily distinct subsets of the $n$-set $X$. We refer to this  colletions of subsets of an $n$-set as a  {\it configuration of  subsets}. We set  $a_{ij}=1$ if $x_j\in X_i$  and we set $a_{ij}=0$ if $x_j\notin X_i$.  The resulting $(0, 1)$-matrix $A=(a_{i,j})$, $i=1,\ldots, m$, $j=1,\ldots, n$ of size $m$ by $n$ is the \it{incidense matrix}
 for the configurations of subsets $X_1, X_2, \dots, X_m$ of the $n$-set $X$. The $1^{\prime}$s in row $\alpha_i$ of
 $A$ display the elements in the subsets $X_i$ and the $1^{\prime}$s in column $\beta_j$ display the ocurrences of the element $x_j$ among the subsets.\\
 Let  $S=\{s_1\ldots, s_n \}$ an $n$-set and $S_1, \ldots, S_m$ be $m$ subsets of the $n$-set $S$ and ${\EuScript A}$ the
 $m\times n$ incidence matrix, for the configuration of subsets $S_1, \ldots, S_m$.
The pair
 \begin{equation}\label{confinc}
  \big(S,  S_i )_{i=1}^m
  \end{equation}
  we call {\it configuration of incidence of $S$}.
 If $(S^{\prime},  S_i^{\prime} \big)_{i=1}^m$, donde $S^{\prime}=\{s^{\prime}_1\ldots, s^{\prime}_n \}$, is other configuration of incidence then they are isomorphic  if and only if there is a bijection $$\psi: S\longrightarrow S^{\prime}$$
 $$ \psi(s_i)=s_i^{\prime}$$ such that $\psi(S_i)=S_i^{\prime}$ 
 for all $i=1, \ldots, m$ and note that  ${\EuScript A}={\EuScript A}^{\prime}$ where ${\EuScript A}$ and  ${\EuScript A}^{\prime}$ are $(m\times n)$-incidence matrices.\\
 Let $(S^{\prime},  S_i^{\prime} \big)_{i=1}^m$ an incidence configuration then using the cartesian product we define \begin{equation}\label{confcart}
 j \times (S,  S_i )_{i=1}^m:=\bigg( \{ j \}\times S, \; \{ j \}\times S_i\bigg)_{i=1}^m 
 \end{equation}
an incidence configuration.
\begin{lemma}\label{incconfisomr}
The incidence configuration $\{ j \} \times (S,  S_i )_{i=1}^m$ is isomorphic to the configuration $(S,  S_i )_{i=1}^m$
and have the same incidence matrix
\end{lemma} 
\begin{proof}
Given that $|\{ j \} \times S|= |S|$ and $|\{ j \} \times S_i |=|S_i |$ for all  $i=1, \ldots, m$ then the projection mapping
$$\psi: \{ j \} \times S \longrightarrow S $$    $$(j,s) \longmapsto s$$ 
is one-one and clearly $\psi_{|\{ j  \}\times S_i}=S_i$ so the configurations are isomorphic. Furthermore, we have
$(j, s)\in \{ j \}\times S_i$ si y solo si $s\in S_i$ and thus both configurations have the same incidence matrix.
\end{proof}
 The following lemma is a direct consequence of his hypotheses.
 \begin{lemma}\label{matrixsubconf001}
  Let $S$ be an $n$-set,  $(S, S_i)_{i \in \Omega}$ and  $(S, S_i)_{i\in \Omega^{\prime}}$ two incidence configurations  such that $\Omega^{\prime}\subseteq \Omega$ and let $ {\EuScript A}_{\Omega}$,  ${\EuScript A}_{\Omega^{\prime}}$ 
  their respective configuration matrices then ${\EuScript A}_{\Omega^{\prime}}$ is a submatrix of $ {\EuScript A}_{\Omega}$.
 \end{lemma}
 \qed
 \begin{corollary}\label{sumconfs}
 Let $S$ be an $n$-set,  $(S, S_i)_{i \in \Omega}$, $(S, S_i)_{i\in \Omega_1}$ and $(S, S_i)_{i\in \Omega_2}$   incidence configurations. If  $\Omega=\Omega_1\cup \Omega_2$ and $\Omega_1\cap \Omega_2=\emptyset$ then      ${\EuScript A}_{\Omega}\sim{\EuScript A}_{\Omega_1}\oplus {\EuScript A}_{\Omega_2}$ up to row permutation and where ${\EuScript A}_{\Omega}$, 
 ${\EuScript A}_{\Omega_1}$ and ${\EuScript A}_{\Omega_2}$ are the respective configuration matrices.
 \end{corollary}
 \begin{proof}
 By lemma \ref{matrixsubconf001} ${\EuScript A}_{\Omega_1}$ and ${\EuScript A}_{\Omega_2}$ son submatrices de 
 ${\EuScript A}_{\Omega}$ also by hypothesis we have that $\Omega=\Omega_1\cup \Omega_2$ and $\Omega_1\cap \Omega_2=\emptyset$ then clearly 
 $$ {\EuScript A}_{\Omega} \sim \left[ \begin{tabular}{c|c}
 ${\EuScript A}_{\Omega_1}$ & $ 0 $ \cr\hline $0$  & ${\EuScript A}_{\Omega_2}$
\end{tabular} \right] $$ and so ${\EuScript A}_{\Omega}\sim{\EuScript A}_{\Omega_1}\oplus {\EuScript A}_{\Omega_2}$ 
 up to row permutation.
 \end{proof}
 \subsection{Matrices}
 A $(0,1)$-matrix is a matrix in which each element is either $0$ or is $1$. A  $(k, \ell)$-matrix is a $(0, 1)$-matrix with $k$ ones in each row and $\ell$ ones in each column, for uses and applications of this type of matrices see \cite{bib 0.1},  \cite{bib 2},  \cite{bib 5a} and  \cite{bib 5}. A sparse matrix is a $(0,1)$-matrix if many of its elements are zero. There are two broad types of sparse matrices: structured and unstructure see \cite[Chapter 3 ]{bib 5.001}  and they have many properties and are applied to different areas of mathematics, such as the theory of error-correcting detector codes, see \cite{bib 3.5}, \cite{bib 3} and \cite{bib 4}. 
A matrix is of \it{binomial order} if its order is a product of two binomial numbers.
\begin{definition}\label{fractusmatrix01}
Let $k\geq 2$ and $\ell\geq 2$ be arbitrary integers, we say that $A$ is a {\it fractal matrix} if it is
a matrix which is built with a recursive algorithm, $(k,\ell)$-matrix, tetra-fragmented, sparse, of binomial order and where $B$ and $C$ are also matrix fractal. 
\end{definition}
\subsection{Geometry}
Let $E$ a finite-dimensional vector space over a field ${\mathbb F}$.\\
Let us consider the following binary relation defined in $E-\{ 0\}$.\\
$$ x\sim y \;\; if\; and\; only\; if\;exists\; \lambda\in {\mathbb F}^*\;such\;that\; y=\lambda x$$ 
The set of equivalence classes
\begin{equation}
{\mathbb P}(E)=\frac{E}{\sim}
\end{equation}
 receives the name of \textit{ Projective Space} deduced from $E$ see \cite{bib 2.23}. Also of   \cite{bib 2.23} if $A\subset {\mathbb F}[x_1, \ldots, x_s]$ of homogeneous polynomials,   then 
$Z(A)=\{ [x]\in  {\mathbb P}^s: p(x)=0 \;for\;all\;p\in A \}$ denotes the set of zeros of $A$ in  projective space $ {\mathbb P}^s$.\\
 Let $E$ a vector space with a base $\{ e_1,\ldots, e_{2n}\}$   and $2\leq k \leq n$, denote  by $\nwedge^kE$ the
$k$-th exterior power of $E$, which is  generated as vector space by $\{ e_{\alpha}:=e_{\alpha_1}\wedge\cdots \wedge e_{\alpha_k} : \alpha \in I(k,2n)\}$. For $w=\sum_{\alpha\in
I(k, 2n)}p_{\alpha}e_{\alpha} \in \nwedge^kE$, the coefficients
$p_{\alpha}$ are the {\it Pl\"ucker coordinates of} $w$, see  \cite[pag42]{bib 6}, we denote by $G(k, E)$ to the Grassmannian, the set of vector subspaces of dimension $k$ of $E$. The Grassmannian $G(k, E)$ is a algebraic variety of dimension $k(2n-k)$ and can be embedded in a projective space ${\mathbb P}^{c-1}$, where $c=\binom{2n}{k}$ by Pl\"ucker embedding.
The {\it Pl\"ucker embedding} is the injective mapping
\begin{equation}\label{embeddingpluck}
 \rho:G(k, E)\rightarrow {\mathbb P}(\wedge^{k}E)
 \end{equation}
  given on each $W\in G(k, E)$ by choosing  a basis
$w_1,\ldots, w_k$ of $W$ and then mapping the vector subspace
$W\in G(k, E)$ to the tensor $w_1\swedge\cdots\swedge w_k\in
\nwedge^kE$. Since choosing a different basis  of $W$ changes the
tensor $w_1\swedge\cdots\swedge w_k$ by a nonzero scalar, this
tensor is a well-defined element in the projective space ${\mathbb
P}(\nwedge^kE)\simeq{\mathbb P}^{N-1}$, where $N=C^{2n}_k=\dim_{\mathbb{F}}(\nwedge^kE)$. 
If $w = \sum_{\alpha \in I(k, 2n)}P_\alpha e_\alpha \in
{\mathbb{P}}(\wedge^k E)$, then $w\in G(k, 2n)$ if and only if for
each pair of tuples $1\leq \alpha_1< \cdots <
\alpha_{k-1}\leq 2n$ and $1\leq \beta_1< \cdots < \beta_{k+1}\leq 2n$, the Pl\"ucker coordinates of $w$ satisfy the quadratic \it{Pl\"ucker relation}
\begin{equation}\label{eq1.1}
 Q_{\alpha, \beta}:=\sum_{i=1}^{k+1} (-1)^i X_{\alpha_1\cdots \alpha_{k-1}\beta_i}X_{\beta_1, \beta_2\cdots
\widehat{\beta_i} \cdots \beta_{k+1}} \in {\mathbb F}[X_{\alpha}]_{\alpha\in I(k, 2n)}
\end{equation}
where  $\widehat{\beta_i}$ means that the corresponding term is omitted and where $\alpha=(\alpha_1\cdots \alpha_{k-1}) \in I(k-1, 2n)$,  $\beta=(\beta_1, \beta_2\cdots
\beta_i \cdots \beta_{k+1}) \in I(k+1, 2n)$, see \cite[section 4]{bib 6}. 
\subsection{Isotropic Grassmannian}
Following \cite{bib 0.001}, let $E$ be  an $2n$-dimensional  vector space over an arbitrary field
${\mathbb F}$ with a nondegenerate, skew-symmetric, standard  bilinear of the  form $\langle x, y \rangle=\sum_{i=1}^n [(x_i\cdot y_{2n+i-1})-(x_{2n+1-i}\cdot y_i)]$. Then there is a basis ${\mathfrak{B}}=\{ e_1, \ldots ,e_{2n}\}$, of $E$ such that
\begin{equation}\label{basesym700}
 \langle\; e_i , e_{j}\;\rangle=\begin{cases}
  1    & \text{if $j=2n-i+1$ for $1\leq i\leq n$}, \\
   0  & \text{otherwise}.
\end{cases}
\end{equation}
we will call this base {\it  simplectic base}. 
Here we say that  
\begin{equation}\label{simplspace11}
E=(E, \langle\; ,\;\rangle)
\end{equation}
is a symplectic vector space of dimension $2n$.
Recall that a vector subspace $W\subseteq E$ is {\it isotropic} if $\langle
x,y\rangle=0$ for all $x,y\in W$, and if $W$ is isotropic its
dimension is at most $n$. For an integer $2\leq k\leq n$, the $k$-th {\it Isotropic Grassmannian}
is the set
\begin{gather*}
IG(k,E)=\{W\in G(k,E): W\;\text{is isotropic of dimension } k \},
\end{gather*}
where as before $G(k,E)$ denotes the Grassmannian variety of vector subspaces
of dimension $k\leq n$ of $E$ and parameterizes all isotropic vector subspaces of dimension $k$ of $E$. So
\begin{equation}\label{tonal1}
IG(k,E)=\{[w_1\wedge\cdots\wedge w_k]\in G(k, E): \langle w_i,w_j\rangle=0\;\text{for all $1\leq i<j\leq k$}\}
\end{equation}
If we do $k=n$ then $IG(n, E)$ parameterizes all maximal simplectic subspaces of $E$ and we call it Lagrangian-Grassmannian de $E$ and we denote  by $L(n, E)$.\\
The following lemma is well known see \cite[pag. 8]{bib 0.001}:\\
\begin{lemma}\label{cannaslem}
\begin{description}
\item[ a)] For any vector space $E$, the direct sum $V=E\oplus E^*$ has a canonical symplectic structure determined by the formula above $$\Omega_0(u\oplus \alpha, v\oplus \beta)=\beta(u)-\alpha(v) $$
\item[b)] If $Y$ is a lagrangian subspace, $(V, \Omega)$ is symplectomorphic to the space $(Y\oplus Y^*, \Omega_0)$, where $\Omega_0$ is determined by the formula above.
\end{description}
\end{lemma}
\qed
 \subsection{Contraction Map}\label{prelimin2}
 Let $E$ be  an $2n$-dimensional  vector space over an arbitrary field
${\mathbb F}$ with a nondegenerate, skew-symmetric bilinear form $\langle\; ,\;\rangle$.   
Consider  a basis ${\mathfrak{B}}=\{ e_1, \ldots ,e_{2n}\}$ as in  \ref{basesym700}.
Then we define the {\it contraction map }
$f:\wedge^kE\rightarrow \wedge^{k-2}E$  given by
\begin{equation}\label{cont-map}
f(w_1\wedge\cdots\wedge w_k)=\sum_{1\leq r<s\leq k}\langle w_r,w_s\rangle(-1)^{r+s-1}
w_1\wedge\cdots\wedge \widehat{w}_r\wedge\cdots\wedge
\widehat{w}_s\wedge\cdots\wedge w_k 
\end{equation}
where $\widehat{w}$ means
that the corresponding term is omitted. 
\begin{proposition}\label{tonal12}
Let $E$ symplectic vector space of dimension $2n$, over an arbitrary field ${\mathbb F}$ with a nondegenerate, skew-symmetric bilinear form $\langle\; ,\;\rangle$,  then  $$IG(k, E)=G(k, E)\cap \ker f$$
\end{proposition}
\qed\\
Now for the given basis ${\mathfrak{B}}=\{ e_1, \ldots ,e_{m}\}$ of the vector space $E$, and for
$\alpha=(\alpha_1,\ldots,\alpha_k)\in I(k,2n)$, write
\begin{align*}\label{nottation}
e_{\alpha}&:=e_{\alpha_1}\swedge\cdots\swedge e_{\alpha_k},\\
e_{\alpha_{rs}}&:=e_{\alpha_1}\swedge\cdots\swedge\widehat{e}_{\alpha_r}
\swedge\cdots\swedge\widehat{e}_{\alpha_s}\swedge\cdots   \swedge
e_{\alpha_k}\\
P_{\alpha}&=(P_{\alpha_1}, \ldots, P_{\alpha_k})
\end{align*}
\begin{theorem}\label{Pi-alpha} Let $E$ symplectic vector space of dimension $2n$, over an arbitrary field
${\mathbb F}$ with  a  base as in \ref{basesym700} and a nondegenerate, skew-symmetric bilinear form $\langle\; ,\;\rangle$. Let $w\in \wedge^kE$
written in Pl\"ucker coordinates as $w=\sum X_{\alpha}e_{\alpha}$, then
\begin{equation}\label{eq2.1}
 f(w)=\sum_{1\leq r<s\leq k}\bigg(\sum_{i=1}^nX_{i\alpha_{rs}(2n-i+1)}\bigg)e_{\alpha_{rs}}
\end{equation} 
\end{theorem}
\begin{proof}
Let $w=\sum_{\alpha\in I(k-2, 2n)} X_{\alpha}e_{\alpha} \in \wedge^kE$,
then for all  positive integers $\beta_1,\beta_2$
such that  $\beta_1 + \beta_2=2n+1$ we have
\begin{align*}
 f(w) & =\sum_{\alpha\in I(k-2, 2n)} X_{\alpha}f(e_{\alpha})\\
 & = \sum_{\alpha\in I(k,2n)}X_{\alpha} \bigg(\sum_{1\leq r<s\leq
k}\langle e_{\alpha_r},e_{\alpha_s}\rangle (-1)^{r+s-1}e_{\alpha_1}\wedge\cdots\wedge
\widehat{e}_{\alpha_r}\wedge\cdots\wedge \widehat{e}_{\alpha_s}\wedge\cdots\wedge
e_{\alpha_k}\bigg)\\
& = \sum_{1\leq r<s\leq k } \bigg(\sum_{\beta_1+\beta_2=2n+1}\langle e_{\beta_1}, e_{\beta_2}\rangle(-1)^{\beta_1+\beta_2-1}
X_{{\alpha}_{{rs}\beta_1\beta_2}}\bigg)e_{\alpha_{rs}}\\
& =\sum_{1\leq r<s\leq k}
\bigg(\sum_{i=1}^nX_{i\alpha_{rs}(2n-i+1)}\bigg)e_{\alpha_{rs}}
\end{align*}
where the next to last equality is because $\beta_1+\beta_2=2n+1$. 
\end{proof}
\begin{corollary}\label{Pi-alpha-rs}  Let $E$ symplectic vector space of dimension $2n$, over an arbitrary field
${\mathbb F}$ with  a  base as in \ref{basesym700} and a nondegenerate, skew-symmetric bilinear form $\langle\; ,\;\rangle$.  Then for $w=\sum X_{\alpha}e_{\alpha}\in \wedge^kE$ written in Pl\"ucker coordinates, we have that
\begin{equation}\label{eq2.1}
 w\in \ker f\Leftrightarrow\sum_{i=1}^kX_{i\alpha_{rs}(2n-i+1)}=0 \;\; \text{\rm for all} \;\; \alpha_{rs}\in I(k-2,2n)
\end{equation} and $2\leqslant k \leqslant n$.
\end{corollary}
\begin{proof}
From the theorem \ref{Pi-alpha} we have
$$\sum_{1\leq r<s\leq k}
\big(\sum_{i=1}^kX_{i\alpha_{rs}(2n-i+1)}\big)e_{\alpha_{rs}}=0\quad \text{if and only if}\quad \sum_{i=1}^{k}X_{i\alpha_{rs}(2n-i+1)}=0$$ 
for all $\alpha_{rs} \in I(k-2, 2n)$.
\end{proof}
Now, for $\alpha_{st}\in I(k-2,2n)$ define the following homogeneous linear polynomials in ${\mathbb F}[X_{\alpha}]_{\alpha\in I(k, 2n)}$
\begin{equation}\label{eq3.1}    
\Pi_{\alpha_{st}}:=\sum_{i=1}^kc_{i,\alpha_{rs},2n-i+1}X_{i,\alpha_{rs},2n-i+1}
\end{equation}
with
\begin{equation}\label{coefPi}
c_{i,\alpha_{rs},2n-i+1}=\begin{cases}
1  &  \text{if $|supp\{i,\alpha_{rs},2n-i+1\}|=k$}, \\
0 & \text{otherwise},
\end{cases}
\end{equation}
We denote by
\begin{equation}\label{matrassIG}
B_{f}
\end{equation}
  the matrix of order $\binom{2n}{k-2}\times \binom{2n}{k}$ associated to the system of linear equations
\begin{equation}\label{matrassIG2}
\{ \Pi_{\alpha_{rs}}=0: \alpha_{rs}\in I(k-2, 2n)\}
\end{equation}
Let $E$ be a symplectic  vector space of dimension $2n$ then by the corollary \ref{Pi-alpha-rs}  it is easy to see that the kernel of the matrix $B_f$ is
\begin{equation}\label{2nkerB}
\ker B_f=\{ [X_{\alpha}]_{\alpha\in I(k, 2n)}\in {\mathbb X}^{C_k^{2n}-1}: \sum_{i=1}^{k}X_{i\alpha_{rs}(2n-i+1)}=0\}
\end{equation}
Also $\ker B_f$ is a vector space of dimension $C_{2n}^{k}-rank B_f$ and so then we have the following corollary
\begin{corollary}\label{tonal12101}
Let $E$ be a vector space of dimension $2n$ and let $f$ be the contraction map then $\ker B_f\simeq \ker f$
\end{corollary}
\begin{proof}
If we restrict  \ref{embeddingpluck} to $\ker f$ we have
$$ \rho : \ker f\longrightarrow \ker B_f$$
$$w=\sum_{\alpha\in I(k-2, 2n)} X_{\alpha}e_{\alpha}  \longmapsto (X_{\alpha})_{\alpha_{rs}\in I(k-2, 2n)} $$
by \ref{2nkerB} it is well defined and  is a linear transformation, by the corollary \ref{Pi-alpha-rs} we have that $\rho$ is an epimorphism, also $\dim \ker f=  C_{k}^{2n}-rank B_f=\dim \ker B_f $. So $\rho$ is a isomorphism.
\end{proof}

 If $Q_{\alpha, \beta}\in {\mathbb F}[X_{\alpha}]_{\alpha\in I(k, 2n)}$ denote the
quadratic Pl\"ucker polynomials, as in \ref{eq1.1}, that define the Grassmann variety
$G(k,E)$, now in terms of 
$\Pi_{\alpha_{st}}$, we have the following
characterization of $IG(k,E)$, as the common zeros of the quadratic
polynomials $Q_{\alpha, \beta}$ and of
$\Pi_{\alpha_{st}}$, that is
\begin{equation}\label{linearsection001}
IG(k, E)= Z\langle Q_{\alpha^{\prime}, \beta^{\prime}},  \Pi_{\alpha_{rs}}\rangle\subset {\mathbb P}^{C^{2n}_k-1}
 \end{equation}
 where $\alpha^{\prime}\in I(n-1,2n), \; \beta^{\prime}\in I(n+1,2n)$ and $\alpha_{rs}\in I(k-2,2n)$. 
\begin{definition}\label{pluckerMatrix1001}
Let $E$  vector space of dimension $2n$, over an arbitrary field ${\mathbb F}$  then we call the matrix $B_f$ the Pl\"ucker matrix of the Isotropic Grassmannian $IG(k, E)$ for $2\leq k \leq n$.
\end{definition} 
\begin{theorem}\label{th211m}
\begin{description}
\item[a)] 
 Let $E$  symplectic vector space of dimension $2n$, over an arbitrary field ${\mathbb F}$. Then  
the Isotropic Grassmannian 
\begin{equation}\label{linearsection011}
IG(k, E)= Z\langle Q_{\alpha^{\prime}, \beta^{\prime}}\rangle\cap \ker B_f
 \end{equation}
 for $2\leq k \leq n$.
 \item[b)]  If $E$ and $E^{\prime}$ are symplectic vector space if $E$ symplectomorphic to $E^{\prime}$ then ambos spaces have the same matrix Pl\"ucker $B_f$
\end{description}
\end{theorem}
\begin{proof}
For $a)$ it follows from the proposition \ref{tonal12} and from the corollary \ref{tonal12101},
 $b)$ it is easy.
\end{proof}
\section{incidence configurations} 
We define the incidence configuration
\begin{equation}\label{incidekisot}
(I(k, 2n),S_{\alpha_{rs}})_{\alpha_{rs}\in I(k-2, 2n)}
\end{equation}
 where, $I(k, 2n)$ is a $C^{2n}_{k}$-set,
 \begin{align}\label{incidekisot2}
 S_{\alpha_{rs}}&=\{ \beta\in I(k, 2n): |supp \{ \alpha\}\cap supp \{ \beta\}\cap \Sigma_n|=k-1\}\\
                         &=\{ (\alpha_{rs}, P_i)\in I(k-2, 2n)\times \Sigma_n: c_{\alpha_{rs}P_i}=1\}
\end{align}
where  $c_{\alpha_{rs}P_i}$ is equal to \ref{coefPi},
 clearly the incidence matrix of \ref{incidekisot} is $B_{f}$ defined in  \ref{matrassIG} .\\
Let  $2\leq k \leq n$ even integers,   and $\Sigma_n=\{P_1, \ldots, P_n \}$ see 
\begin{equation}\label{Setconf1}
\big( C_{\frac{k}{2}}(\Sigma_n), \;S_{P_{\alpha}} \big)_{P_{\alpha}\in C_{(k-2)/2}(\Sigma_n)}
\end{equation} 
where $C_{\frac{k}{2}}(\Sigma_n)$ is a $C_{\frac{k}{2}}^{n}$-set and 
\begin{equation}\label{Setconf2}
S_{P_{\alpha}}=\{  P_{\beta}\in C_{\frac{k}{2}}(\Sigma_n): supp\{\alpha\}\subset supp\{ \beta\} \}
\end{equation}
a configuration of subsets of $C_{\frac{k}{2}}(\Sigma_n)$, such that $P_{\alpha}\in C_{\frac{k-2}{2}}(\Sigma_n)$.
%
\begin{lemma}\label{onesroow}
 For all $P_{\alpha}\in C_{\frac{k-2}{2}}(\Sigma_n)$ we have $|S_{P_{\alpha}}|=n-\frac{k-2}{2}$ 
\end{lemma}
\begin{proof}
If $P_{\beta}\in S_{P_{\alpha}}$ and $|supp\{\alpha\}|=\frac{k-2}{2}$ then 
\begin{align*}
|S_{P_{\alpha}}|&=|\{\beta\in I(k/2, n): supp\{\beta \}=supp\{\alpha \}\cup\{ i\}\; with \; 1\leq i \leq n\}|\\
        &=|\{i\in [n] : |supp \{\alpha \}\cup\{ i\} |=k/2\}|\\
        &=n-|supp\{\alpha \}|\\
        &=n-\frac{k-2}{2}      
\end{align*}
\end{proof}
\begin{lemma}\label{lem1210}
Let $P_\alpha$ and $P_{\overline{\alpha}}$ two different elements of $C_{\frac{k-2}{2}}(\Sigma_n)$ then 
$$S_{P_{\alpha}}\cap S_{P_{\overline{\alpha}}}\neq \emptyset \;\; if\; and\; only\; if\;\; |supp\{\alpha \}\cap supp \{ \overline{\alpha}\}|=\frac{k-4}{2}$$
\end{lemma}
\begin{proof}
$ \Rightarrow):$ Let $ P_{\beta}\in  S_{P_{\alpha}}\cap  S_{ P_{\overline{\alpha}}}$ then there are two different positive integers $N$ and $M$ different such that
$supp \{ \alpha \} \cup \{ N\}=supp \{ \beta\} =supp \{ \overline{\alpha}\} \cup \{ M\}$, also 
$supp \{ \alpha\} - \{ M\}=supp \{ \beta\}-\{N, M \} =
supp \{\overline{\alpha}\} - \{ N\}$, as a consequence  we have to $supp \{\beta\}-\{N, M \}= supp \{\alpha\} \cap supp \{ \overline{\alpha}\}$ so
$|supp\{P_{\alpha}\}\cap supp\{ P_{\overline{\alpha}} \} |=|supp \{ P_{\beta}\}|-|\{N, M \}|=\frac{k}{2}-2=\frac{k-4}{2}$\\
$\Leftarrow):$ Suppose $ |supp\{ \alpha\}\cap supp\{\overline{\alpha} \} |=\frac{k-4}{2}$  then there exist $M$, $N$ distinct positive integers such that $\{M\}= supp \{ \alpha\}- supp \{ \alpha\}\cap supp \{\overline{ \alpha}\}$ and  $\{N\}= supp \{\overline{ \alpha}\}- supp \{ \alpha\}\cap supp \{\overline{ \alpha}\}$ and so it exists
 $P_{\beta}\in C_{\frac{k}{2}}(\Sigma_n)$ such that 
$supp \{ \beta\}=(supp \{ \alpha\}\cap supp \{ \overline{\alpha}\})\cup \{ N, M\}$ then it is easy to see 
$P_{\beta}\in S_{ P_{\alpha}}\cap  S_{ P_{\overline{\alpha}}}$ and so $S_{P_{\alpha}}\cap S_{P_{\overline{\alpha}}}\neq \emptyset$.
\end{proof}
\begin{corollary}\label{canes1}
Let $P_{\alpha}$ and $P_{\overline{\alpha}}$ be two different of $C_{\frac{k-2}{2}}(\Sigma_n)$ then $|S_{P_{\alpha}}\cap S_{P_{\overline{\alpha}}}|\leq 1$
\end{corollary}
\begin{proof}
Let $ P_{\beta}\in S_{ P_{\alpha}}\cap  S_{P_{\overline{\alpha}}}$     
then there are two positive integers $N$ and $M$ such that
$supp \{ \alpha \} \cup \{ N\}=supp \{ \beta\} =supp \{ \overline{\alpha}\} \cup \{ M\}$, also 
$supp \{ \alpha\} - \{ M\}=supp \{ \beta\}-\{N, M \} =
supp \{\overline{\alpha}\} - \{ N\}$, as a consequence we have to $supp \{\beta\}= (supp \{\alpha\} \cap supp \{ \overline{\alpha}\})\cup \{N, M\} $  clearly 
$(supp \{\alpha \}\cap supp \{ \overline{\alpha}\})\cup \{M\}\subset supp \{\alpha \}$
then for  lemma \ref{lem1210} we have 
$ |(supp\{  \alpha\}\cap supp\{ \overline{\alpha} \})\cup \{M\} |=\frac{k-4}{2}+1=\frac{k-2}{2}$ so  we have to $supp \{\alpha \}=(supp \{\alpha \}\cap supp \{ \overline{\alpha}\})\cup\{M \}$, analogously we have  to $supp \{\overline{\alpha} \}=(supp \{\alpha \}\cap supp \{ \overline{\alpha}\})\cup\{N\}$. Suppose there is another  $P_{\beta^{\prime}}\in S_{P_{\alpha}}\cap  S_{ P_{\overline{\alpha}}}$ then there exist $M^{\prime}$ y $N^{\prime}$ distinct positive integers such that  
$supp\{\beta^{\prime}\}=supp \{\alpha \}\cap supp \{ \overline{\alpha}\}\cup \{N^{\prime}, M^{\prime}\}$, as 
$supp\{\alpha \}\subseteq supp\{ \beta^{\prime}\}$ then $(supp \{\alpha\}\cap supp \{\overline{\alpha}\})\cup \{ M \} \subseteq (supp \{\alpha\}\cap supp \{\overline{\alpha}\})\cup \{  N^{\prime}, M^{\prime} \} $ then $M\in \{ N^{\prime}, M^{\prime}\}$. Analogamente $supp\{\overline{\alpha} \}\subseteq supp\{ \beta^{\prime}\}$ and so we have to $N\in \{ N^{\prime}, M^{\prime}\}$ so $\{N, M\}=\{N^{\prime}, M^{\prime}\}$ then  $supp\{\beta \}=supp\{\beta^{\prime}\}$ 
and  so by \ref{suppiguls} $ |S_{P_{\alpha}}\cap  S_{ P_{\overline{\alpha}}}|\leq 1$
\end{proof}
\begin{corollary}\label{canes2}
If $S_{P_{\alpha}} = S_{P_{\overline{\alpha}}}$ then $P_{\alpha}=P_{\overline{\alpha}}$
\end{corollary}
\begin{proof}
Suppose $S_{ P_{\alpha}} = S_{P_{\overline{\alpha}}}$ and  
$ P_{\alpha}\neq  P_{\overline{\alpha}}$ then 
 by corollary 
\ref{canes1} we have to $ |S_{ P_{\alpha}}\cap  S_{ P_{\overline{\alpha}}}|\leq 1$ however this is not possible since by lemma \ref{onesroow}, $ |S_{ P_{\alpha}}\cap  S_{ P_{\overline{\alpha}}}|=|S_{ P_{\alpha}}|=n-\frac{k-2}{2}$  
but $n\geq k \geq 2$ and so $\frac{2(n+1)-k}{2}\geq 2$ which implies that $P_{\alpha}=P_{\overline{\alpha}}$
\end{proof}
%
%
\begin{lemma}\label{unoscolumn}
Each $S_{P_{\alpha}}$ has nonempty interseccion  with exactly $k/2$ of the other subsets.
\end{lemma}
\begin{proof}
Clearly  $P_{\beta}\in S_{ P_{\alpha}}$ 
if and only if  $\alpha\in C_{\frac{k-2}{2}}\{supp\{ \beta\} \}$, so the number subsets $S_{ P_{\alpha}}$ that contain  $ P_{\beta}$ is equal to $$|C_{(\frac{k-2}{2})}\{supp\{ \beta\}\}|=C_{\frac{k-2}{2}}^{k/2}=k/2$$ 
\end{proof}
\begin{proposition}\label{princmatrx}
Let $2\leq k \leq n$ even positive integer  then the incidence matrix of \ref{Setconf1}, satisfies
\begin{description}
\item[a)] has $ n-\frac{k-2}{2}$-ones in each row
\item[b)]  has $k/2$-ones in each column
\item[c)]  every two lines have at most one $1$ in common
\item[d)]   is sparse matrix
\end{description}
\end{proposition}
\begin{proof}
For $a)$ the $1's$ in each row are displayed by the elements in the subset $S_{P_{\alpha}}$ so by lemma \ref{onesroow} each row has exactly $ n-\frac{k-2}{2}$ ones in each row.\\
For $b)$ the $1's$ in each column  are display the occurrences of the elements of $S_{P_{\alpha}}$ among the subsets, this follows from the lemma \ref{unoscolumn}.\\
$c)$ follows directly from corollary \ref{canes1}.\\
$d)$ The density of ones in the matrix and is given by  $\lfloor (2(n+1)-k)/2\rfloor\times C^k_{\frac{k-2}{2}}=(k/2)\times C^k_{\frac{k}{2}}$
 $$\frac{ (2(n+1)-k)/2}{ C^k_{k/2}}=\frac{k/2}{ C^k_{\frac{k-2}{2}}}$$
which approaches zero as $k$ approaches infinity .
\end{proof}

Let $2\leq k \leq n$ with $k$ even integer and $n$ arbitrary integer consider the incidence configuration \ref{Setconf1}
we denote the $C_{\frac{k-2}{2}}^n\times C_{\frac{k}{2}}^n$-incidence matrix
\begin{equation}\label{Lmatrix}
{\EuScript A}_{n-\frac{k-2}{2} }^{ k/2}
\end{equation}
If in  \ref{Lmatrix} we make $k=n$  then we denote by 
\begin{equation}\label{Lmatrix2}
{\EuScript L}_{r}={\EuScript A}_{ \frac{n+2}{2} }^{ n/2}
\end{equation}
where $r=\frac{n+2}{2}$ so the order is $C_{\frac{n-2}{2}}^n\times C_{\frac{n}{2}}^n$-incidence matrix
 Let $ n$ and $k$ even integers such that $2\leq k \leq n$, $r:=\frac{n+2}{2}$,  $1\leq \ell \leq r-1$ 
and   $1\leq a_1<a_2<\cdots<a_{2\ell}\leq 2n$ such that $a_i+a_j\neq 2n+1$  then we define 
$\Sigma_{a_1,\ldots,a_{2\ell}}=\Sigma_n-\{P_{a_1}, \dots, P_{a_{2\ell}}\}$ as in \ref{setSigma0011}.
We define an cartesian incidence configuration as in \ref{confcart}
  \begin{equation}\label{subsetconf12}
 (a_1,\ldots, a_{2\ell})\times  \bigg( C_{\frac{k-2\ell}{2}}(\Sigma_{a_1,\ldots ,a_{2\ell}}),  \; S_{(a_1,\ldots, a_{2\ell}, P_{\alpha})}
  \bigg)_{P_{\alpha}\in C_{\frac{k-2(\ell+1)}{2}}(\Sigma_{a_1,\ldots,a_{2\ell}})}
  \end{equation}
 where 
  \begin{equation*}\label{subsetconf21}
 \{(a_1,\ldots, a_{2\ell})\}\times C_{\frac{k-2\ell}{2}}(\Sigma_{a_1,\ldots ,a_{2\ell}}) 
\end{equation*}
is an $C^{n-2\ell}_{\frac{k-2\ell}{2}}$-set and the subsets are
\begin{equation}\label{subsetconf212}
  S_{ (a_1,\ldots, a_{2\ell}, P_{\alpha})}:=\{ (a_1,\ldots, a_{2\ell}, P_{\beta}) :
 supp\{\alpha\}\subset supp\{\beta\} \}
 \end{equation} 
 for all $P_{\alpha}\in C_{\frac{k-2(\ell+1)}{2}}(\Sigma_{a_1,\ldots,a_{2\ell}})$.
\begin{lemma}\label{isomconf1}
Let $n$ even integer, $k$ integer such that $4\leq k \leq n$, $r:=\frac{n+2}{2}$,  $1\leq \ell \leq r-1$ 
and   $1\leq a_1<a_2<\cdots<a_{2\ell}\leq 2n$ such that $a_i+a_j\neq 2n+1$   then the $C_{\frac{k-2(\ell+1)}{2}}^{n-2\ell}\times C_{\frac{k-2\ell}{2}}^{n-2\ell}$ incidence matrix of \ref{subsetconf12} is  ${\EuScript A}_{n-\frac{k+2(\ell-1)}{2}}^{(k-2\ell)/2}$. 
\end{lemma}
\begin{proof}
As $|\Sigma_{a_1,\ldots,a_{2\ell}}|=n-2\ell$ if we make ${\overline n}=n-2\ell$ and ${\overline k}=k-2\ell$
then renumber the elements of $\Sigma_{a_1,\ldots,a_{2\ell}}= \{ P_1, \ldots, P_{\overline{n}}\}$ then by the lemma \ref{incconfisomr} the configuration
$$(a_1,\ldots, a_{2\ell})\times \bigg(  C_{\frac{k-2\ell}{2}}(\Sigma_{a_1,\ldots ,a_{2\ell}}), S_{ P_{\alpha}} \bigg)_{P_{\alpha}\in C_{\frac{k-2(\ell+1)}{2}}(\Sigma_ {a_1,\ldots,a_{2\ell}})}$$
 is isomorphic to the incidence configuration
  $$\big( C_{\frac{{\overline k}}{2}}(\Sigma_{{\overline n}}), S_{P_{\alpha}} \big)_{P_{\alpha}\in C_{\frac{{\overline k} -2}{2}}(\Sigma_{{\overline n}})}$$
  so also for the lemma  \ref{incconfisomr} both  have the same incidence matrix 
  \begin{align*}
  {\EuScript A}_{\overline{n}-\frac{\overline{k}-2}{2}}^{\frac{\overline{k}}{2}}&={\EuScript A}_{(n-2\ell)-\frac{(k-2\ell)-2}{2}}^{\frac{k-2\ell}{2}}\\
                                                                                &={\EuScript A}_{n-\frac{k+2(\ell-1)}{2}}^{\frac{k-2\ell}{2}}
  \end{align*}
  \end{proof}
 We denote the incidence matrix of \ref{subsetconf12} as
 \begin{equation}\label{matrizcomp}
 (a_1,\ldots, a_{2\ell}){\EuScript A}_{n-\frac{k-2(\ell+1)}{2}}^{(k-2\ell)/2}:={\EuScript A}_{n-\frac{k+2(\ell-1)}{2}}^{(k-2\ell)/2}
 \end{equation}
For $ n\geq 3$ odd  and $j \in \{1, \ldots, n \}$, 
we define an cartesian incidence configuration as in \ref{confcart}
\begin{equation} \label{subsetconf122}
j\times \bigg( C_{\frac{k-1}{2}}(\Sigma_n-\{P_j\}), S_{ P_{\alpha}}
 \bigg)_{P_{\alpha}\in C_{\frac{k-3}{2}}(\Sigma_n)}
\end{equation}
where
 \begin{equation*}
\{j\}\times C_{\frac{k-1}{2}}(\Sigma_n-\{P_j\}) 
\end{equation*} 
is a $C^n_{\frac{k}{2}}$-set and its subsets $S_{(j, P_{\alpha})}$ define by
  \begin{equation}\label{prodcruz0001}
  S_{(j, P_{\alpha})}=\{ (j, P_{\beta}): supp\{\alpha\}\subset supp\{ \beta\}\}\\                         
  \end{equation}
  with $P_{\alpha}\in C_{\frac{n-3}{2}}(\Sigma_n)$\\  
   For $ n\geq 5$ odd integer, $k$ integer such that $2 \leq k\leq n$, $r:=\frac{n+1}{2}$ and consider
  $1\leq a_1<a_2<\cdots<a_{2\ell+1}\leq 2n$ such that $a_i+a_j\neq 2n+1$   
  we define
 $\Sigma_{a_1,\ldots,a_{2\ell+1}}=\Sigma_n-\{P_{a_1}, \ldots, P_{a_{2\ell+1}} \} $ \\
 note that $|\Sigma_{a_1,\ldots,a_{2\ell+1}}|=n-(2\ell+1)$.
We define an cartesian incidence configuration as in \ref{confcart}
  \begin{equation}\label{subsetconf2}
   (a_1,\ldots, a_{2\ell+1})\times  \bigg ( C_{\frac{k-(2\ell+1)}{2}}(\Sigma_{a_1,\ldots ,a_{2\ell+1}}),  S_{ (a_1,\ldots, a_{2\ell+1}, P_{\alpha})} \bigg)_{P_{\alpha}\in C_{\frac{k-(2\ell+3)}{2}}(\Sigma_{a_1\ldots, a_{2\ell+1}})}
  \end{equation}
where 
  \begin{equation*}
 \{ (a_1,\ldots, a_{2\ell+1})\} \times C_{\frac{k-(2\ell+1)}{2}}(\Sigma_{a_1,\ldots ,a_{2\ell+1}}) \}
\end{equation*}
is $C^{n-(2\ell+1)}_{\frac{k-(2\ell+1)}{2}}$-set
where the  family of subsets is given by
\begin{equation}
 S_{(a_1,\ldots, a_{2\ell+1}, P_{\alpha})}=\{ (a_1,\ldots, a_{2\ell+1}, P_{\beta}):
 supp\{\alpha\}\subset supp\{\beta\ \}\\
 \end{equation} 
for all $P_{\alpha}\in C_{\frac{k-2(\ell+3)}{2}}(\Sigma_{a_1\ldots, a_{2\ell+1}})$ and 
$P_{\beta}\in C_{\frac{k-2(\ell+1)}{2}}(\Sigma_{a_1\ldots, a_{2\ell+1}})$
\begin{lemma}\label{isomconf2}
\begin{description}
\item[a])
The incidence matrix \ref{subsetconf122} is ${\EuScript A}_{(n-1)-\frac{k-3}{2}}^{ (k-1)/2}$ 
\item[ b)] 
The incidence matrix \ \ref{subsetconf2} is ${\EuScript A}^{((k-1)-2\ell)/2}_{[(n-1)-2\ell]-\frac{(k-3)-2\ell}{2}}
$  
\end{description}
\end{lemma}
\begin{proof}
For the proof of $a)$ let $m=k-1$
$$j\times \bigg( C_{\frac{k-1}{2}}(\Sigma_n-\{P_j\}),\;  S_{P_{\alpha}}  \bigg)_{P_{\alpha}\in C_{\frac{k-3}{2}}(\Sigma_n-\{P_j \})}$$ is isomorphic to the incidence configuration
$$ \bigg( C_{\frac{m}{2}}(\Sigma_n-\{P_j\}),\; S_{P_{\alpha}} \bigg)_{P_{\alpha}\in C_{\frac{m-2}{2}}(\Sigma_n-\{P_j\})} $$ 
so its configuration matrix is given by 
$${\EuScript A}_{(n-1)-\frac{m-2}{2}}^{ m/2}={\EuScript A}_{(n-1)-\frac{k-3}{2}}^{ (k-1)/2}$$
For $b)$ clearly $|\Sigma_{a_1,\ldots, a_{2\ell+1}}|=n-(2\ell+1)$ if we do ${\overline n}=n-(2\ell+1)$, renaming the elements of $\Sigma_{a_1,\ldots, a_{2\ell+1}}=\{P_1, \ldots, P_{{\overline n}} \}$ 
and we do ${\overline k}=k-(2\ell + 1)$ then
$$ (a_1,\ldots, a_{2\ell+1})\times\bigg ( C_{\frac{k-(2\ell+1)}{2}}(\Sigma_{a_1,\ldots ,a_{2\ell+1}}), S_{P_{\alpha}} \bigg)_{P_{\alpha}\in C_{\frac{k-(2\ell+3)}{2}}(\Sigma_{a_1\ldots, a_{2\ell}})}$$
is isomorphic to the incidence configuration $$\bigg( C_{\frac{{\overline k}}{2}}(\Sigma_{{\overline n}}), S_{P_{\alpha}} \bigg)_{P_{\alpha}\in C_{\frac{\overline k}{2}}(\Sigma_{\overline n})} $$  so its incidence matrix is
\begin{align*}
{\EuScript A}^{\overline{k}/2}_{\overline{n}-\frac{\overline{k}-2}{2}}&={\EuScript A}^{(k-(2\ell+1))/2}_{n-(2\ell+1)-\frac{k-(2\ell+1)-2}{2}}\\
                                                                                                       &={\EuScript A}^{((k-1)-2\ell)/2}_{[(n-1)-2\ell]-\frac{(k-3)-2\ell}{2}}
\end{align*}
 \end{proof}
 For each $j\in\{1, \ldots, n \}$ we denote by 
  \begin{equation}\label{noddj}
  \{ j \}{\EuScript A}
  \end{equation}
  to the incidence matrix of  \ref{subsetconf122}.\\
Also for each $(a_1, \ldots, a_{\ell+1})\in C_{\frac{k-2(\ell+1)}{2}}(\Sigma_{a_1, \ldots, a_{\ell+1}})$ such that
$a_i+a_j\neq 2k+1$ we denote by 
\begin{equation}\label{nodd1}
(a_1, \ldots, a_{\ell+1}){\EuScript A}_{[(n-1)-2\ell]-\frac{(k-3)-2\ell}{2}}^{((k-1)-2\ell)/2}
\end{equation}
to the matrix  \ref{subsetconf2}.
\begin{lemma}\label{lem2.5} 
\begin{description}
\item[A) ] 
If $4\leq k \leq n$,  $r=\lfloor\frac{k+2}{2}\rfloor$, $1\leq \ell \leq r-2$  and $j\in \{1, \ldots, n \}$ then   
{ \begin{equation}\label{In2npartc11}
 C_{\lfloor\frac{k-2}{2}\rfloor}(\Sigma_{n})\cup\Big(\bigcup_{\ell=1}^{r-2}
\bigcup_{\substack{1\leq a_1<\cdots < a_{2\ell}\leq 2n\\ a_i+a_j\neq 2n+1}}(a_1,\ldots, a_{2\ell})\times C_{\lfloor \frac{k-2(\ell+1)}{2}\rfloor}(\Sigma_{a_1,\ldots ,a_{2\ell}})\Big).
\end{equation}}
is a partition of the set $I(k-2,2n)$.
\item[B)]
If $5\leq k \leq n$ and let $r=\lfloor\frac{k+1}{2} \rfloor$,   $1\leq \ell \leq r-2$, then   
{ \begin{equation}\label{In2npartc12}
\bigcup_{j=1}^{n}\{j\}\times C_{\lfloor\frac{k-2}{2}\rfloor}(\Sigma_n)\cup\Big(\bigcup_{\ell=1}^{r-2}
\bigcup_{\substack{1\leq a_1<\cdots < a_{2\ell+1}\leq 2k\\ a_i+a_j\neq 2k+1}}(a_1,\ldots, a_{2\ell+1})\times C_{\lfloor\frac{k-2(\ell+1)}{2}\rfloor}(\Sigma_{a_1,\ldots ,a_{2\ell+1}})\Big).
\end{equation}}
is a partition of the set $I(k-2,2n)$.
\end{description}
\end{lemma}
\begin{proof}
Let $k\geq 4$ and let $r=\lfloor\frac{k+2}{2}\rfloor$,  then   it is sufficient to show that $I(k-2, 2n)$ is contained in \ref{In2npartc11} $($ resp. If $k\geq 5$ and let $r=\lfloor\frac{k+1}{2}\rfloor$,  then  $I(k-2, 2n)$ is contained in \ref{In2npartc12} $)$. If $supp\{ \alpha_{rs}\}\subseteq \Sigma_n$ then it exists $P_{\theta}\in C_{\lfloor\frac{k-2}{2}\rfloor}(\Sigma_n)$ such that
$\alpha_{rs}=P_{\theta}\in C_{\lfloor\frac{k-2}{2}\rfloor}(\Sigma_n)$ ( resp. exists $P_{\theta}\in C_{\lfloor \frac{k-3}{2}\rfloor}(\Sigma_n)$ such that
$\alpha_{rs}=(j,P_{\theta})\in \{j \}\times C_{\lfloor\frac{k-3}{2}\rfloor}(\Sigma_n)$). 
If $supp\{ \alpha_{rs}\}\cap \Sigma_n= \emptyset$ then $\alpha_{rs}=(a_1, \ldots, a_{\ell-2})$ such that $a_i+a_j\neq 2k+1$  and so we have to $\alpha_{rs}\in C_0(\Sigma_{a_1, \ldots, a_{\ell-2}})$ ( resp. $\alpha_{rs}=(a_1, \ldots, a_{\ell-3})\in C_0(\Sigma_{a_1, \ldots, a_{\ell-3}})$ since $a_i+a_j\neq 2k+1$).   
If $supp\{ \alpha_{rs}\}\cap \Sigma_n\neq \emptyset$ then there are  $P_{\theta}\in C_{\lfloor\frac{n-2(k+1)}{2}\rfloor}(\Sigma_{a_1, \ldots, a_{2\ell}})$ 
and $1\leq a_1<\ldots, < a_{2\ell}\leq 2n$ such that $\alpha_{rs}=(a_1, \ldots, a_{2\ell}, P_{\theta})\in (a_1, \ldots, a_{2\ell})\times C_{\lfloor\frac{n-2k}{2}\rfloor}(\Sigma_{a_1, \ldots, a_{2\ell}})$ .( resp. exist $P_{\theta}\in C_{\lfloor\frac{(n-1)-2k}{2}\rfloor}(\Sigma_{a_1, \ldots, a_{2\ell+1}})$ and $1\leq a_1<\ldots, < a_{2\ell+1}\leq 2n$ such that $\alpha_{rs}=(a_1, \ldots, a_{2\ell+1}, P_{\theta})\in (a_1, \ldots, a_{2\ell+1})\times C_{\lfloor\frac{(n-1)-2k}{2}\rfloor}(\Sigma_{a_1, \ldots, a_{2\ell+1}})$ )
with which the demonstration ends.
\end{proof}
\begin{theorem}\label{sumdirect}
Let $E$ symplectic vector space of dimension $2n$, over an arbitrary field ${\mathbb F}$ with a nondegenerate, skew-symmetric bilinear form $\langle\; ,\;\rangle$, and let $B_f$ the P\"ucker matrix of Isotropic Grasmmannian $IG(k, E)$
\begin{description}
\item[A) ] 
 $k\geq 4$ even integer, $r=(k+2)/2$ and $1\leq \ell \leq r-2$ then\\
$B_{f}={\EuScript A}^{k/2}_{n-\frac{k-2}{2}}\oplus\bigg(\bigoplus_{\ell=1}^{r-2} \bigoplus_{\substack{1\leq a_1<\cdots<a_{2\ell}\leq 2n\\ a_i+a_j\neq 2n+1}}( a_1, \ldots, a_{2\ell}){\EuScript A}^{(k-2\ell)/2}_{n-\frac{k+2(\ell-1)}{2}}\bigg)$ \\
\item[B)]
$k\geq 5$ odd integer, $r=(k+1)/2$ and $1\leq \ell \leq r-2$ then\\
$B_{f}=({\EuScript A}^{\lfloor k/2 \rfloor}_{n-\lfloor \frac{k-2}{2}\rfloor})^n\oplus\bigg(\bigoplus_{\ell =1}^{r-2} \bigoplus( a_1, \ldots, a_{2\ell+1}){\EuScript A}^{((k-1)-2\ell)/2}_{[(n-1)-2\ell]-\frac{(k-3)-2\ell}{2}}\bigg)$ \\
\end{description}
where  $1\leq a_1<\cdots<a_{2\ell+1}\leq 2n$,  $a_i+a_j\neq 2n+1$ and $({\EuScript A}^{\lfloor k/2 \rfloor}_{n-\lfloor \frac{k-2}{2}\rfloor})^n:=\bigoplus_{j=1}^{n}\{j\}{\EuScript A}^{\lfloor k/2 \rfloor}_{n-\lfloor \frac{k-2}{2}\rfloor}$
\end{theorem}
\begin{proof}
$A):$  $k\geq 2$ even integer, $r=(k+2)/2$ and $1\leq \ell \leq r-2$ then
consider the incidence configuration \ref{incidekisot} $(I(k, 2n) , S_{\alpha_{rs}})_{\alpha_{rs}\in I(k-2, 2n)}$
Consider the  incidence matrices given in \ref{matrizcomp}, \ref{noddj} y \ref{nodd1}.\\
Now  the partition given in \ref{lem2.5}\\
$I(k, 2n)=C_{\frac{k-2}{2}}(\Sigma_{n})\cup\Big(\bigcup_{\ell=1}^{r-2}
\bigcup_{\substack{1\leq a_1<\cdots < a_{2\ell}\leq 2\ell\\ a_i+a_j\neq 2k+1}}(a_1,\ldots, a_{2\ell})\times C_{\frac{k-2(\ell+1)}{2}}(\Sigma_{a_1,\ldots ,a_{2\ell}})\Big)$\\
So, we have  \\
$(I(k, 2n), S_{\alpha_{rs}})_{\alpha_{rs}\in C_{\frac{k-2}{2}}(\Sigma_{n})\cup\Big(\bigcup_{\ell=1}^{r-2}
\bigcup_{\substack{1\leq a_1<\cdots < a_{2\ell}\leq 2\ell\\ a_i+a_j\neq 2k+1}}(a_1,\ldots, a_{2\ell})\times C_{\frac{k-2(\ell+1)}{2}}(\Sigma_{a_1,\ldots ,a_{2\ell}})\Big)}$\\
consequently by the corollary  \ref{sumconfs}
we have that the incidence matrix  satisfies\\
$B_{f}={\EuScript A}^{k/2}_{n-\frac{k-2}{2}}\oplus\bigg(\bigoplus_{i=1}^{r-2} \bigoplus_{\substack{1\leq a_1<\cdots<a_{2\ell}\leq 2n\\ a_i+a_j\neq 2n+1}}( a_1, \ldots, a_{2\ell}){\EuScript A}^{(k-2\ell)/2}_{n-\frac{k-2(\ell+1)}{2}}\bigg)$ \\
$B)$:  $k\geq 3$ odd integer, $r=(k+1)/2$ and $1\leq \ell \leq r-2$ then\\
Considering the incidence configuration \ref{incidekisot}  and the partition of $I(n-2, 2n)$ given in \ref{In2npartc12}, 
consequently by the corollary  \ref{sumconfs}\\
$B_{f}=({\EuScript A}^{\lfloor k/2 \rfloor}_{n-\lfloor \frac{k-2}{2}\rfloor})^n\oplus\bigg(\bigoplus_{i=1}^{r-2} \bigoplus_{\substack{1\leq a_1<\cdots<a_{2\ell+1}\leq 2n\\ a_i+a_j\neq 2n+1}}( a_1, \ldots, a_{2\ell+1}){\EuScript A}^{(k-1)-2\ell)/2}_{[(n-1)-2\ell]-\frac{(k-3)-2\ell}{2}}\bigg)$ \\
\end{proof}
\begin{ejemplo}\label{example0001}
Let $E$ be a simplectic vector space of dimension 14 and $f:\wedge^7E\longrightarrow \wedge^5E$ contraction map then
    \begin{align*}
    B_f  = \bigoplus_{i=1}^7 {\EuScript L}_4^{\{i\}} \oplus \Big(
    \bigoplus_{1\leq a_1 < a_2< a_3 \leq 14 \atop a_i + a_j \neq 15 }
    (a_1 a_2 a_3){\EuScript L}_3 \Big) \oplus \Big( \bigoplus_{1\leq
        a_1< a_2<\ldots <a_4<a_5\leq 14 \atop a_i + a_j \neq 15} (a_1a_2 a_3a_4 a_5){\EuScript
        L}_2 \Big),
    \end{align*}
    where ${\EuScript L}_4^{\{i\}} = {\EuScript L}_4$, $(a_1a_2a_3){\EuScript
        L}_3 = {\EuScript L}_3$ and $(a_1a_2a_3a_4a_5){\EuScript
        L}_2 = {\EuScript L}_2$.  
$$ \begin{aligned}
    B_f= \begin{pmatrix}
    \begin{tabular}{|c c c c c c c c c c}
    \cline{2-1} \multicolumn{1}{c}{}  &
    \multicolumn{1}{|c|}{${\EuScript L}_4$} & & \multicolumn{1}{c}{} &
    & & \multicolumn{1}{c}{$0$}& \\
    \cline{2-2}\multicolumn{1}{c}{} & \multicolumn{1}{c}{} &
    \multicolumn{1}{c}{$ \ddots$} & & \multicolumn{1}{c}{ }  & \\
    \cline{4-3}\multicolumn{1}{c}{ } &\multicolumn{1}{c}{ } &
    \multicolumn{1}{c|}{ }&\multicolumn{1}{c|}{${\EuScript L}_4$}  & & \\
    \cline{5-3}\cline{4-4} \multicolumn{1}{c}{} & \multicolumn{1}{c}{}
    & \multicolumn{1}{c}{}& \multicolumn{1}{c|}{} &
    \multicolumn{1}{c|}{${\EuScript L}_3$}   & \\
    \cline{5-5} \multicolumn{1}{c}{} & \multicolumn{1}{c}{$0$} &
    \multicolumn{1}{c}{}& \multicolumn{1}{c}{} & \multicolumn{1}{c}{}
    & \multicolumn{1}{c}{$\ddots$}
    \\ \cline{7-7}
    \multicolumn{1}{c}{} & \multicolumn{1}{c}{} &
    \multicolumn{1}{c}{}& \multicolumn{1}{c}{ } &
    \multicolumn{1}{c}{ } & \multicolumn{1}{c|}{ } &   \multicolumn{1}{c|}{${\EuScript L}_3$} \\
    \cline{7-7}\cline{8-8} \multicolumn{1}{c}{} & \multicolumn{1}{c}{}
    & \multicolumn{1}{c}{} & \multicolumn{1}{c}{}& \multicolumn{1}{c}{
    } &
    \multicolumn{1}{c}{ } & \multicolumn{1}{c|}{ } &   \multicolumn{1}{c|}{${\EuScript L}_2$} \\
    \cline{8-8}  \multicolumn{1}{c}{} & \multicolumn{1}{c}{} &
    \multicolumn{1}{c}{} & \multicolumn{1}{c}{} &
    \multicolumn{1}{c}{}& \multicolumn{1}{c}{} & \multicolumn{1}{c}{}
    &\multicolumn{1}{c}{} & \multicolumn{1}{c}{$\ddots$} \\
    \cline{10-10}  \multicolumn{1}{c}{} & \multicolumn{1}{c}{} &
    \multicolumn{1}{c}{} & \multicolumn{1}{c}{} &
    \multicolumn{1}{c}{}& \multicolumn{1}{c}{} & \multicolumn{1}{c}{}
    &\multicolumn{1}{c}{} & \multicolumn{1}{c|}{} &
    \multicolumn{1}{c|}{${\EuScript L}_2$} \\ \cline{10-10}
    \end{tabular} \end{pmatrix}
    \end{aligned}
$$
\end{ejemplo}
\section{ build  fractal matrix} \label{Secc2}
Let $M$ be a matrix of order $n\times m$, the operation ${\EuScript O}(M)$ is to paste to the matrix $M$ at the bottom, the identity matrix $I_{m\times m}$, which generates a matrix  of order $(n+m)\times m$, that is 
${\EuScript O}(M) =\left(\begin{smallmatrix}
M  \\ I_{C_{\ell -1}^{k+\ell -2}} 
        \end{smallmatrix}\right) $.
 If we have a matrix vector $V=(M_1, \ldots, M_t)$  the operation 
${\EuScript O}(V)$ is the matrix vector $({\EuScript O}(M_1), \ldots, {\EuScript O}(M_t))$. 
Let $M_1$ and $M_2$ be two matrices of order $n_1\times m_1$ and   $n_2\times m_2$ respectively where $n_1\geq n_2$, 
{\it paste the matrix concatenatedly to the right}, side by side, is to get the matrix ${\EuScript P}(M_1, M_2)=M_1\bigsqcup M_2$  of order $n_1\times (m_1+m_2)$, where  $\sqcup$ means
joining together side-by-side and aligning the bottoms of the
corresponding identity matrices and filling the non-marked spaces on
the upper right blocks with zeroes. Finally denote by $\underline{k}$  the order matrix $1\times k$ filled with 1's.
\begin{equation}\label{Algoritm1}
 Algorithm\; I
\end{equation}
\hrulefill
\begin{description}
\item[{\rm Input: }] $k$ and $\ell$ arbitrary positive integers.
\item[{\rm Output: }] \; The  $A_k^{\ell}$ matrix, with $k$ ones in each row and $\ell$ ones in each column  
\end{description}
\hrulefill
\begin{description}
\item[{\rm Step 1.}]\; Let $V=(\underline{k}, \underline{k-1}, \ldots, \underline{2}, \underline{1})$ be matrix vector.
\item[{\rm Step 2.}] \;We apply the operation ${\EuScript O}(V)$ to the matrix vector given in Step 1 and  we obtain   
the matrix vector $({\EuScript O}(\underline{k}), {\EuScript O}(\underline{k-1}), \ldots, {\EuScript O}(\underline{2}), {\EuScript O}(\underline{1}))$.
\item[{\rm Step 3.}]\; Put ${\EuScript P}({\EuScript O}(\underline{k}), {\EuScript O}(\underline{k-1}), \ldots, {\EuScript O}(\underline{2}), {\EuScript O}(\underline{1}))$ we generate a matrix that we denote by $A_k^2$,
 with $k$ ones in each row and $2$ ones in each column
\item[{\rm Step 4.}]\; Now consider the matrix vector \\ $V=(A_k^2, A_{k-1}^2,\ldots, A_2^2,A_1^2)$.
\item[{\rm Step 5.}]\; Returne to the Steps 2,  3,  4  with matrix vector $V=(A_k^2, A_{k-1}^2,\ldots, A_2^2,A_1^2)$
to build a matrix $A_k^3$  given by ${\EuScript P}({\EuScript O}(A_k^2), {\EuScript O}(A_{k-1}^2),\ldots, {\EuScript O}(A_2^2), {\EuScript O}(A_1^2))$,  with $k$ ones in each row and $3$ ones in each column. \\
Consider vector $V=(A_k^3, A_{k-1}^3,\ldots, A_2^3,A_1^3)$
\item[{\rm Step 6.}] \;The algorithm ends when when the number of ones in each column is equal to $\ell$
\end{description}
\hrulefill 
\begin{definition}\label{Ass-Matrix}
 Let  $k$ and $\ell$ be arbitrary positive integers,  we define the matrices  $A_k^{\ell}$  inductively as
 $A_k^2:={\EuScript P}({\EuScript O}(\underline{k}), {\EuScript O}(\underline{k-1}), \ldots, {\EuScript O}(\underline{2}), {\EuScript O}(\underline{1}))$ and \\
 $A_k^3:= {\EuScript P}({\EuScript O}(A_k^2), {\EuScript O}(A_{k-1}^2),\ldots, {\EuScript O}(A_2^2), {\EuScript O}(A_1^2))$, that is 
  $$ A_k^3=A_k^2 \left(I_{C^{k+1}_2}
    \right)\sqcup A_{k-1}^2 \left( I_{C^{k}_2} \right)\sqcup \cdots \sqcup
    A_2^2\left( I_{C^3_2} \right)\sqcup A_1^2 \left( I_{C^2_2} \right).$$
 With the previous notation we define the following matrices
 \begin{align*}
    A_k^4&=A_k^3\left( I_{C^{k+2}_3}
    \right) \sqcup A_{k-1}^3 \left( I_{C^{k+1}_3} \right) \sqcup \cdots \sqcup
    A_2^3\left( I_{C^4_3} \right) \sqcup A_1^3\left( I_{C^3_3} \right),\\
    A_k^5&=A_k^4\left( I_{C^{k+3}_4}
    \right) \sqcup A_{k-1}^4 \left( I_{C^{k+2}_4} \right)\sqcup \cdots \sqcup
    A_2^4\left( I_{C^5_4} \right)\sqcup A_1^4 \left( I_{C^4_4}\right),\\
    \vdots & \quad\quad\quad\quad\quad\quad\quad \quad\quad\quad\vdots \\
  A_k^{\ell} & =A_k^{\ell-1}\left( I_{C^{k+\ell-2}_{\ell-1}}
\right)\sqcup  \cdots \sqcup A_2^{\ell-1}\left(
I_{C^{k-1}_{k-2}} \right) \sqcup A_1^{\ell-1}\left(
I_{C^{k-2}_{k-2}} \right). 
  \end{align*}
 \end{definition}
 \begin{ejemplo}
If in the previous algorithm we do $k=4$ and $\ell=3$ the algorithm produces
the matrix of $15\times 20=300$ inputs,  $A_4^3$ has only 60 ones, so its density is $\frac{60}{300}=0.2$.
  $$ A_4^3= \left( \begin{tabular}{cccccccccccccccccccc}
1 & 1 & 1 & 1 & 0 & 0 & 0 & 0 & 0 & 0 & 0 & 0 & 0 & 0 & 0 & 0 & 0 &
0 & 0 & 0 \cr \cline{1-4}\multicolumn{1}{|c} 1 & 0 & 0 & 0 &
\multicolumn{1}{|c}{$1$} & 1 & 1 & 0 & 0 & 0 & 0 & 0 & 0 & 0 & 0 & 0
& 0 & 0 & 0 & 0 \cr \cline{5-7}\multicolumn{1}{|c} 0 & 1 & 0 & 0 &
\multicolumn{1}{|c}{$1$} & 0 &  0 & \multicolumn{1}{|c} 1 & 1 & 0 &
0 & 0 & 0 & 0 & 0 & 0 & 0 & 0 & 0 & 0 \cr
\cline{8-9}\multicolumn{1}{|c} 0 & 0 & 1 & 0 & \multicolumn{1}{|c} 0
& 1 & 0 & \multicolumn{1}{|c} 1 & 0 & \multicolumn{1}{|c} 1 & 0 & 0
& 0 & 0 & 0 & 0 & 0 & 0 & 0 & 0 \cr \cline{10-10}\multicolumn{1}{|c}
0 & 0 & 0 & 1 & \multicolumn{1}{|c} 0 & 0 &  1 & \multicolumn{1}{|c}
0 & 1 & \multicolumn{1}{|c} 1 & \multicolumn{1}{|c} 0 & 0 & 0 & 0 &
0 & 0 & 0 & 0 & 0 & 0 \cr \cline{1-10}\multicolumn{1}{|c} 1 & 0 & 0
& 0 & 0 & 0 & 0 & 0 & 0 & 0 & \multicolumn{1}{|c} 1 & 1 & 1 & 0 & 0
& 0 & 0 & 0 & 0 & 0 \cr \cline{11-13}\multicolumn{1}{|c} 0 & 1 & 0 &
0 & 0 & 0 & 0 & 0 & 0 & 0 &\multicolumn{1}{|c} 1 & 0 & 0 &
\multicolumn{1}{|c} 1 & 1 & 0 & 0 & 0 & 0 & 0 \cr
\cline{14-15}\multicolumn{1}{|c} 0 & 0 & 1 & 0 & 0 & 0 & 0 & 0 & 0 &
0 & \multicolumn{1}{|c} 0 & 1 & 0 & \multicolumn{1}{|c} 1 & 0 &
\multicolumn{1}{|c} 1 & 0 & 0 & 0 & 0 \cr \cline{16-16}
\multicolumn{1}{|c} 0 & 0 & 0 & 1 & 0 & 0 & 0 & 0 & 0 & 0 &
\multicolumn{1}{|c} 0 & 0 & 1 & \multicolumn{1}{|c} 0 & 1 &
\multicolumn{1}{|c} 1 & \multicolumn{1}{|c} 0 & 0 & 0 & 0 \cr
\cline{11-16}\multicolumn{1}{|c} 0 & 0 & 0 & 0 & 1 & 0 & 0 & 0 & 0 &
0 & \multicolumn{1}{|c} 1 & 0 & 0 & 0& 0 & 0 & \multicolumn{1}{|c} 1
& 1 & 0 & 0 \cr \cline{17-18} \multicolumn{1}{|c} 0 & 0 & 0 & 0 &  0
& 1 & 0 & 0 & 0 & 0 &\multicolumn{1}{|c}  0 & 1 & 0 & 0& 0 & 0 &
\multicolumn{1}{|c} 1 & 0 & \multicolumn{1}{|c} 1 & 0 \cr
\cline{19-19}\multicolumn{1}{|c} 0 & 0 & 0 & 0 & 0 & 0 & 1 & 0 & 0 &
0 & \multicolumn{1}{|c} 0 & 0 & 1 & 0& 0 & 0 & \multicolumn{1}{|c} 0
& 1 & \multicolumn{1}{|c} 1 & \multicolumn{1}{|c} 0 \cr
\cline{17-19}\multicolumn{1}{|c} 0 & 0 & 0 & 0 & 0 & 0 & 0 & 1 & 0 &
0 & \multicolumn{1}{|c} 0 & 0 &  0 & 1& 0 & 0 & \multicolumn{1}{|c}
1 & 0 & 0 & \multicolumn{1}{|c} 1 \cr
 \cline{20-20}\multicolumn{1}{|c} 0 & 0 & 0 & 0 & 0 & 0 & 0 & 0 & 1 &
0 & \multicolumn{1}{|c} 0 & 0 & 0 & 0& 1 & 0 & \multicolumn{1}{|c} 0
& 1 & 0 & \multicolumn{1}{|c|} 1 \cr
\cline{20-20}\multicolumn{1}{|c} 0 & 0 & 0 & 0 & 0 & 0 & 0 & 0 & 0 &
1 & \multicolumn{1}{|c} 0 & 0 & 0 & 0& 0 & 1 & \multicolumn{1}{|c} 0
& 0 & 1 & \multicolumn{1}{|c|} 1 \cr \cline{1-20}
\end{tabular} \right) \quad $$ \\
\end{ejemplo}
\begin{lemma}
For each $k$, and $\ell$ the matrix $A_k^{\ell}$ it is of order $C_{\ell-1}^{k+\ell-1}\times C_{\ell}^{k+\ell-1}$ 
\end{lemma}
\begin {proof}
The proof is by induction on $k$ and $\ell$.
Clearly the matrices $A_1^1=1$,  $A_k^1$, $A_1^{\ell}$ and $A_2^2$ they meet the formula.
Suppose for all $k$ and $\ell$ and for everything $k^{\prime}<k$  we have the order of $A_k^{\ell^{\prime}}$ is $C_{\ell^{\prime}-1}^{k+\ell^{\prime}-1}\times C_{\ell^{\prime}}^{k+\ell^{\prime}-1}$ and for all $k^{\prime}< k$  the order of $A_{k^{\prime}}^{\ell}$ is $C_{\ell-1}^{k^{\prime}+\ell-1}\times C_{\ell}^{k^{\prime}+\ell-1}$. Then the order of the matrix $A_k^{\ell}$ is given as follows 
$$   \left[ \begin{tabular}{c|c}
$C_{(\ell-1)-1}^{k+(\ell-1)-1}\times C_{\ell-1}^{k+(\ell-1)-1}$ & $ 0 $ \cr\hline $C_{\ell-1}^{k+\ell-2}\times C_{\ell-1}^{k+\ell-2}$  & $C_{\ell-1}^{(k-1)+\ell-1}\times C_{\ell}^{(k-1)+\ell-1}$
\end{tabular} \right]$$ 
in addition, the order of $A_k^{\ell}$ is given by  $$(C_{\ell-2}^{k+\ell-2} + C_{\ell-1}^{k+\ell-2})\times (C_{\ell-1}^{k+\ell-2}\times C_{\ell}^{k+\ell-2})=C_{\ell-1}^{k+\ell-1}\times C_{\ell}^{k+\ell-1}$$
\end{proof}
 \begin{theorem}\label{Fragment-11}
Let $k\geq 2$  and $\ell\geq 2$ arbitrary positive integers then $A_k^{\ell}$ is fractal matrix.
$$ A_k^{\ell} = \left[ \begin{tabular}{c|c}
 $A_k^{\ell-1}$ & $ 0 $ \cr\hline $I_{C_{\ell-1}^{k+\ell-2}}$  & $A_{k-1}^{\ell}$
\end{tabular} \right] $$
\end{theorem}
\begin{proof}
By Algorithm \ref{Algoritm1}  we have that $A_k^{\ell}$ is constructed with a recursive algorithm.
That is
$$ A_k^{\ell} =A_k^{\ell-1}\left( I_{C^{k+\ell-2}_{\ell-1}}
\right)\sqcup   A_{k-1}^{\ell-1}\left(
I_{C^{k+\ell-3}_{\ell-1}} \right) \sqcup\cdots \sqcup A_1^{\ell-1}\left(
I_{C^{\ell-1}_{\ell-1}} \right)$$
By definition \ref{Ass-Matrix} we have
$$A_{k-1}^{\ell}=A_{k-1}^{\ell-1}\left(
I_{C^{k+\ell-3}_{\ell-1}} \right)\sqcup A_{k-2}^{\ell-1}\left(
I_{C^{k+\ell-4}_{\ell-1}} \right)  \sqcup\cdots \sqcup A_1^{\ell-1}\left(
I_{C^{\ell-1}_{\ell-1}} \right) $$
Then
$$ A_k^{\ell} = \left[ \begin{tabular}{c|c}
 $A_k^{\ell-1}$ & $ 0 $ \cr\hline $I_{C_{\ell-1}^{k+\ell -2}}$  & $A_{k-1}^{\ell}$
\end{tabular} \right] $$
\end{proof}
\section{ The matrix   ${\EuScript A}_{n- \lfloor \frac{n-2}{2}\rfloor}^{\lfloor n/2 \rfloor}$ is fractal matrix} 
Let $E$ symplectic vector space of dimension $2n$, over an arbitrary field ${\mathbb F}$ with a nondegenerate, skew-symmetric bilinear form $\langle\; ,\;\rangle$. 
Let $m\geq 8$ an even integer, define
\begin{align*}
[s]&=\{1,\ldots, s\},\quad\text{for $1\leq s\leq m$}\\
[s,m]&=\{P_{s+1},\ldots, P_m\}\\
r&=\frac{m+2}{2}
\end{align*}
For 
\begin{equation}\label{alpha1113}
\overline{\alpha}=(\overline{\alpha}_1,\ldots,\overline{\alpha}_{(m-6)/2})\in I((m-6)/2,m)
\end{equation}
 we define the {\it triangle set}, $T_{\overline{\alpha}}\subseteq C_{\frac{m-2}{2}}(\Sigma_m)$, as 
\begin{equation}\label{tttrian}
T_{\overline{\alpha}}=\big\{ P_{\alpha}\in C_{\frac{m-2}{2}}(\Sigma_m): P_{\alpha}=P_{(\alpha_1, \ldots, \alpha_{\frac{m-6}{2}}, \epsilon_{\frac{m-4}{2}},  \epsilon_{\frac{m-2}{2}})} \big\}
\end{equation}
where $\alpha_{\frac{m-6}{2}+1}\leq \epsilon_{\frac{m-4}{2}}< \epsilon_{\frac{m-2}{2}}\leq m$.
Using \ref{alpha1113} and the cartesian product we can see that the triangle set satisfies
\begin{equation}\label{sumtria111}
T_{(\overline{\alpha}_1,\ldots, \overline{\alpha}_{(m-6)/2})}:=\{P_{(\overline{\alpha}_1,
\ldots, {\overline{\alpha}_{(m-6)/2}})}\}\times
C_2\big([\alpha((m-6)/2),m]\big)
\end {equation}
\begin{lemma}\label{lem2.1}
Let $m\geq 8$ even integer then
$$ C_{\frac{m-2}{2}}(\Sigma_m)=\bigcup_{\alpha\in I((m-6)/2,m)}T_{\alpha}$$
\end{lemma}
\begin{proof}
By  construction  $T_{\alpha}\subseteq C_{\frac{m-2}{2}}(\Sigma_m)$, for each $\alpha \in I((m-6)/2,m)$ so \\ 
$\bigcup_{\alpha\in I((m-6)/2,m-2)}T_{\alpha}\subset C_{\frac{m-2}{2}}(\Sigma_m)$.
For the other containment, let
$P_{(\alpha_1,\ldots, \alpha_{(m-2)/2})}\in C_{\frac{m-2}{2}}(\Sigma_m)$ and hence $(\alpha_1,\ldots, \alpha_{(m-6)/2})\in I((m-6)/2, m)$ and $(\alpha_{(m-4)/2}, \alpha_{(m-2)/2})$ satisfy that
$$\alpha_{(m-6)/2+1}\leq \alpha_{(m-4)/2}<\alpha_{(m-2)/2}\leq m$$
It follows that
\begin{align*}
\{P_{\alpha_1},\ldots, P_{\alpha_{(m-2)/2}} \}& \in \{ P_{(\alpha_1,\ldots, {\alpha_{(m-6)/2}})}\}\times C_2([\alpha_{(m-6)/2+1},m])\\
&\quad =T_{(\alpha_1,\ldots, \alpha_{(m-6)/2})}
\end{align*}
\end{proof}
Let  $m$ be even integer and consider \ref{Setconf1}  with $m=n=k$ then we define the injective function
 $\varphi^m:C_{\frac{m-2}{2}}(\Sigma_m)\rightarrow \{0,1\}^{C_{\frac{m}{2}}(\Sigma_m)}$  the
function given by
 \begin{equation}\label{Sec-Mm}
P_{\alpha}\mapsto \big( \varphi_{P_{\beta}}^m(P_{\alpha}) \big)_{P_{\beta}\in C_{\frac{m}{2}}(\Sigma_m)\ },
\end{equation}    
where 
 $$\varphi_{P_{\beta}}^m(P_{\alpha}) 
 =\begin{cases}
1  & \text{if $P_{\beta}\in S_{P_{\alpha}} $} \\
0& \text{otherwise}.
\end{cases}$$
and $S_{P_{\alpha}}=\{P_{\beta}\in C_{\frac{m}{2}}(\Sigma_m): supp\{ \alpha\} \subset supp \{ \beta\}\}$ \\
\begin{lemma}\label{phiinyct}
$\varphi^m$  injective function and with the notation \ref{Lmatrix2} we have that $\varphi^m(C_{\frac{m-2}{2}}(\Sigma_m))={\EuScript L}_r$
\end{lemma}
\begin{proof}
If $\big( \varphi_{P_{\beta}}^m(P_{\alpha}) \big)=\big( \varphi_{P_{\beta}}^m(P_{\alpha^{\prime}}) \big)$ then it is easy to see $S_{P_{\alpha}}=S_{P_{\alpha^{\prime}}}$ then by corollary \ref{canes2} we have $P_{\alpha}=P_{\alpha^{\prime}}$.\\
Also the matrix of order $C^m_{\frac{m-2}{2}}\times C^m_{\frac{m}{2}}$ generated by $\varphi^m(C_{\frac{ m-2}{2}}(\Sigma_m))$ is, up to row permutation, equal to the matrix ${\EuScript L}_r$.
\end{proof}
  Let  $\displaystyle\overline{\alpha}=(1, 2,\ldots, (m-6)/2 )\in I((m-6)/2,m)$ so we have  
          \begin{equation}\label{unionTriaReng}
    T_{(1, 2,\ldots, \frac{m-6}{2} )}=
       R_{\frac{m-4}{2}}^{\overline{\alpha}} \cup 
        R_{\frac{m-2}{2}}^{\overline{\alpha}} \cup \cdots \cup 
        R_{m-1}^{\overline{\alpha}}
        \end{equation}
    where 
    \begin{align*}
     R^{\overline{\alpha}}_{\frac{m-4}{2}}&:=\big\{ \textstyle  P_{\big(\overline{\alpha}, \frac{m-4}{2}, \frac{m-2}{2}\big)},    P_{\big(\overline{\alpha}, \frac{m-4}{2}, \frac{m}{2}\big)}, \ldots,   P_{ \big(\overline{\alpha}, \frac{m-4}{2}, m\big)} \big\},\\  
 R^{\overline{\alpha}}_{\frac{m-2}{2}}&:=\big\{ \textstyle P_{ \big(\overline{\alpha}, \frac{m-2}{2}, \frac{m}{2}\big)},    P_{\big(\overline{\alpha}, \frac{m-2}{2}, \frac{m+2}{2}\big)}, \ldots,    P_{\big(\overline{\alpha}, \frac{m-2}{2}, m\big)} \big\},\\
  &\; \; \vdots \\
    R^{\overline{\alpha}}_{m-1}&:=\big\{ \textstyle  P_{\big(\overline{\alpha}, (m-1), m\big)} \big\}.
           \end{align*}
we call these sets the rows of $T_{(1, 2,\ldots, \frac{m-6}{2} )}$.
Now consider the  configuration of subsets corresponding to the first row
\begin{lemma} \label{varphi-Row}          
 $\varphi^m(R^{\overline{\alpha}}_{\frac{m-4}{2}})=A_r^2$          
\end{lemma}           
 \begin{proof}
Doing $\overline{\alpha}$ as in \ref{alpha1113} we have:\\
      the  first row of the matrix
    $\varphi^m(T^1)$ is
$$\varphi^m(P_{\big(\overline{\alpha}, \frac{m-4}{2}, \frac{m-2}{2}\big)})=(\overbrace{1,\ldots,1}^r,0,\ldots,0),
$$
Note that $$X_{P_{\big(\overline{\alpha}, \frac{m-4}{2}, \frac{m-2}{2}\big)}}\cap X_{P_{\big(\overline{\alpha}, \frac{m-4}{2}, \frac{m}{2}\big)}}=\big\{ P_{(\overline{\alpha}, \frac{m-4}{2}, \frac{m-2}{2}, \frac{m}{2})}\big\}$$
so  the first two rows of the matrix $\varphi^m(T^1)$ are
    $$\begin{pmatrix}
    \varphi^m(P_{(\overline{\alpha}, \frac{m-4}{2}, \frac{m-2}{2})})  \\
    \varphi^m(P_{(\overline{\alpha}, \frac{m-4}{2}, \frac{m}{2})})
    \end{pmatrix}=\begin{pmatrix}
    \overbrace{1,1,\ldots,1}^r,& 0,\ldots,0,& 0,\ldots,0 \\
    1,0,\ldots,0,&\underbrace{1,\ldots,1}_{r-1},&0,\ldots,0
    \end{pmatrix},
$$
this is justified by
\begin{align*}
X_{P_{(\overline{\alpha}, \frac{m-4}{2}, \frac{m-2}{2})}}\cap X_{P_{(\alpha, \frac{m-4}{2}, \frac{m+2}{2})}}&=\big\{ P_{(\overline{\alpha}, \frac{m-4}{2}, \frac{m-2}{2}, \frac{m+2}{2})}\big\}\\
X_{P_{(\overline{\alpha}, \frac{m-4}{2}, \frac{m}{2})}}\cap X_{P_{(\alpha, \frac{m-4}{2}, \frac{m+2}{2})}}&=\big\{ P_{(\overline{\alpha}, \frac{m-4}{2}, \frac{m}{2}, \frac{m+2}{2})}\big\}
\end{align*}
similarly the first three rows of the matrix $\varphi^m(T_{1}^1)$  are
   $$
    \begin{pmatrix}
    \varphi^m(P_{(\overline{\alpha}, \frac{m-4}{2}, \frac{m-2}{2})})  \\
    \varphi^m(P_{(\overline{\alpha}, \frac{m-4}{2}, \frac{m}{2})})\\
    \varphi^m(P_{(\overline{\alpha}, \frac{m-4}{2}, \frac{m+2}{2})})
    \end{pmatrix}=\begin{pmatrix}
    \overbrace{1,1,\ldots,1}^r,& 0,\ldots,0,& 0,\ldots,0, & 0,\ldots,0 \\
    1,0,\ldots,0,&\underbrace{1,\ldots,1}_{r-1},&0,\ldots,0, & 0,\ldots,0\\
    0,1,\ldots,0,&1,0,\ldots,0&\underbrace{1,\ldots,1}_{r-2},& 0,\ldots,0
    \end{pmatrix},
    $$
    and so on. As  consequence we obtain  
    \begin{align*}
    \varphi^m( R^{\overline{\alpha}}_{\frac{m-4}{2}})&= {\EuScript P}({\EuScript O}(\underline{r}), {\EuScript O}(\underline{r-1}), \ldots, {\EuScript O}(\underline{2}), {\EuScript O}(\underline{1})) \\
    &=A_r^2
    \end{align*} 
    \end{proof}
 \begin{lemma} \label{varphi-Tria}          
  $\varphi^m(T_{1,2, \ldots, \frac{m-6}{2}})=A_r^3$ 
 \end{lemma}   
  \begin{proof}    
    Now, let $I^r$ be the matrix given by the first $r$ rows of the
    identity matrix $I_{C_2^{r+1}}$,  where $\sqcup$ means
joining together side-by-side and aligning the bottoms of the
corresponding identity matrices. It is also easy to see that   
  $\varphi^m(R^{\overline{\alpha}}_{\frac{m-2}{2}})= I^r \sqcup A^2_{r-1}$, consequently we have the following
     \begin{align*} \varphi^m\big(
            R^{\overline{\alpha}}_{\frac{m-4}{2}}\cup 
            R^{\overline{\alpha}}_{\frac{m-2}{2}}\big)& =A_r^2(I^r)\sqcup A_{r-1}^2,\\      
        \varphi^m\big( R^{\overline{\alpha}}_{\frac{m-4}{2}}\cup
       R^{\overline{\alpha}}_{\frac{m-2}{2}}\cup R^{\overline{\alpha}}_{\frac{m}{2}}\big)&=A_r^2(I^{r+(r-1)}) \sqcup \\
        &\sqcup A_{r-1}^2(I^{r-1})\sqcup A_{r-2}^2,\\
      \vdots\quad\quad \quad\quad\quad\quad\quad & \quad\quad\quad \quad\quad\quad\vdots\\ 
        \varphi^m\big( R^{\overline{\alpha}}_{\frac{m-4}{2}}\cup \cdots \cup
         R^{\overline{\alpha}}_{m-2}\big)&
        =A_r^2(I^{r+(r-1)+\cdots+2})\sqcup\\
       &\sqcup A_{r-1}^2(I^{(r-1)+\cdots+2}) \sqcup\\
        &\qquad
        \sqcup\cdots\sqcup A_{2}^2(I^2)\sqcup A_1^2,\\
        \varphi^m\big( R^{\overline{\alpha}}_{\frac{m-4}{2}}\cup \cdots \cup
        R^{\overline{\alpha}}_{m-1}\big)&
        =A_r^2(I^{r+(r-1)+\cdots+1})\sqcup\\
       & A_{r-1}^2(I^{(r-1)+\cdots+1})\sqcup\\
            & \sqcup\cdots\sqcup A_{2}^2(I^3)\sqcup
        A_1^2(I^1),
        \end{align*} 
        where for the last equality we just notice that
    $I^{r+(r-1)+\cdots+1}= I_{C_2^{r+1}}$ and similarly for the other
    $I^t$ in that formula. 
                  
   From the above,  the block stepped matrix
    \begin{gather}\label{imaL} 
\varphi^m(T_{1,2, \ldots, \frac{m-6}{2}}) = A_r^2(I_{C^{r+1}_2})\sqcup A_{r-1}^2(I_{C^{r}_2})\sqcup\cdots\sqcup A_2^2(I_{C^{3}_2})\sqcup A_1^2(I_{C^{2}_2})=A_r^3
\end{gather}
\end{proof}
Recall of \ref{sumtria111} what 
$$T_{\big(1,2,\ldots,\frac{m-8}{2},\frac{m-6}{2}\big)}=\{ P_{(1,\ldots, {(m-8)/2},
{(m-6)/2})}\}\times C_2\{[(m-6)/2,m]\}$$.  Then we recursively define the following sets:
$$\begin{array}{lr}{\begin{array}{l}
    T_{1}^1 =T_{(1,2,\ldots,\frac{m-8}{2},\frac{m-6}{2})}\\
    T_{2}^1 =\bigcup_{i=0}^{\frac{m+2}{2}}T_{(1, 2,\ldots,
        \frac{m-8}{2}, \frac{m-6}{2}+i)}\\
    T_{3}^1 =T_{2}^1\cup\bigcup_{i=0}^{\frac{m}{2}}T_{(1, 3,\ldots,
        \frac{m-6}{2}, \frac{m-4}{2}+i)}\\
    T_{4}^{1} =T_{3}^1\cup \bigcup_{i=0}^{\frac{m-2}{2}}T_{(1,
        4,\ldots,\frac{m-4}{2},\frac{m-2}{2}+i)}\\
    \;\;\,\quad \vdots  \\
    T_{\frac{m+4}{2}}^{1} =T_{\frac{m+2}{2}}^1\cup\bigcup_{i=0}^{1}T_{(1, \frac{m+4}{2},\ldots,m-4,m-3+i)}\\
    T_{\frac{m+6}{2}}^{1} =T_{\frac{m+4}{2}}^1\cup T_{(1,
        \frac{m+6}{2},\ldots,m-3,m-2)}
    \end{array}} &
{\begin{array}{l}
    T_{2}^2 =T_{(2,3,\ldots,\frac{m-6}{2},\frac{m-4}{2})}\\
    T_{3}^2 =\bigcup_{i=0}^{\frac{m}{2}}T_{(2,3,\ldots, \frac{m-6}{2},
        \frac{m-4}{2}+i)}\\
    T_{4}^2 =T_{3}^2\cup\bigcup_{i=0}^{\frac{m-2}{2}}T_{(2,4,\ldots,
        \frac{m-4}{2}, \frac{m-2}{2}+i)}\\
    T_{5}^{2} =T_{4}^2\cup
    \bigcup_{i=0}^{\frac{m-4}{2}}T_{(2,5,\ldots,\frac{m-2}{2},
        \frac{m}{2}+i)}\\
    \;\;\,\quad \vdots \\
    T_{\frac{m+4}{2}}^{2} =T_{\frac{m+2}{2}}^2\cup\bigcup_{i=0}^{1}T_{(2, \frac{m+4}{2},\ldots,m-3+i)}\\
    T_{\frac{m+6}{2}}^{2} =T_{\frac{m+4}{2}}^2\cup T_{(2,
        \frac{m+6}{2},\ldots,m-2)}
    \end{array}}
    \end{array}
 $$
$$\begin{array}{lr}{\begin{array}{l}
    T_{3}^3 = T_{(3,4,\ldots,\frac{m-4}{2},\frac{m-2}{2})}\\
    T_{4}^3 = \bigcup_{i=0}^{\frac{m-2}{2}}T_{(3,4,\ldots, \frac{m-4}{2},\frac{m-2}{2}+i)}\\
    T_{5}^3 = T_{4}^3\cup\bigcup_{i=0}^{\frac{m-4}{2}}T_{(3,5,\ldots,
        \frac{m-2}{2},
        \frac{m}{2}+i)}\\
    T_{6}^{3} = T_{5}^3\cup
    \bigcup_{i=0}^{\frac{m-6}{2}}T_{(3,6,\ldots,\frac{m}{2},
        \frac{m+2}{2}+i)}\\
    \;\;\,\quad \vdots  \hspace{5.3cm} \cdots \\
    T_{\frac{m+4}{2}}^{3} = T_{\frac{m+2}{2}}^3\cup\bigcup_{i=0}^{1}T_{(3,6,\ldots,m-3+i)}\\
    T_{\frac{m+6}{2}}^{3} = T_{\frac{m+4}{2}}^3\cup T_{(3,
        \frac{m+6}{2},\ldots,m-2)}
    \end{array}}  &
{\begin{array}{l}
    T_{\frac{m+2}{2}}^{\frac{m+2}{2}}=T_{(\frac{m+2}{2},\frac{m+4}{2},\ldots,m-4,m-2)}\\
    T_{\frac{m+4}{2}}^{\frac{m+2}{2}}=\bigcup_{i=0}^1T_{(\frac{m+2}{2},\frac{m+4}{2},\ldots,m-4,m-3+i)}\\
    T_{\frac{m+6}{2}}^{\frac{m+2}{2}}=T_{\frac{m+4}{2}}^{\frac{m+2}{2}}\cup T_{(\frac{m+2}{2},\frac{m+6}{2},\ldots,m-4,m-2)}\\
    \quad\qquad \vdots\\
    T_{\frac{m+6}{2}}^{\frac{m+2}{2}}=T_{(\frac{m+4}{2},\frac{m+6}{2},\ldots,m-3,m-2)}.
    \end{array}}
        \end{array}
$$
We define 
\begin{align}
 T_{m}&=T_{\frac{m+6}{2}}^{1}\cup T_{\frac{m+6}{2}}^{2}\cup
\cdots \cup T_{\frac{m+6}{2}}^{\frac{m+4}{2}}\\
T_{m}&=\cup_{i=1}^{\frac{m+4}{2}}T^i _{\frac{m+6}{2}}
\end{align}
By lemma \ref{lem2.1} it is easy to see that 
\begin{equation}\label{UnionTriag001}
C_{\frac{m-2}{2}}(\Sigma_m)=T_m
\end{equation}
Considering similar arguments to those given in the proof of the lemma \ref{varphi-Tria} we have that

\begin{align*}
\varphi^m(T_{1,2, \ldots, \frac{m-6}{2}})&=A_r^3\\
\varphi^m(T_{1,2, \ldots, \frac{m-6}{2}}\cup T_{1,2, \ldots, \frac{m-4}{2}})&=A_r^3(\Gamma^{C^r_1})\cup A_{r-1}^3\\
\varphi^m(T_{1,2, \ldots, \frac{m-6}{2}}\cup T_{1,2, \ldots, \frac{m-4}{2}}\cup T_{1,2, \ldots, \frac{m-2}{2}})&=A_r^3(\Gamma^{C^r_1+C^r_2})\cup A_{r-1}^3(\Gamma^{C^r_2})\sqcup  A_{r-2}^3 \\
&\vdots\\
\varphi^m(T_2^1)&=
        A_r^3(I_{C_3^{r+2}})\sqcup  \cdots\sqcup A_1^3(I_{C_3^{3}})\\
        &= A_r^4
\end{align*}
\begin{theorem}\label{TheorAr{r-1}}
If $m\geq 8$ then ${\EuScript L}_r=A_r^{r-1}$
\end{theorem}
\begin{proof}
Using \ref{imaL} and similar arguments to the proof of lemma \ref{varphi-Tria} we have 
        \begin{align*}
        \varphi^m(T^1_1)&=A_r^2(I_{C_2^{r+1}})\sqcup A_{r-1}^2(I_{C_2^{r}})\sqcup \cdots\sqcup A_2^2(I_{C_2^{3}})\sqcup A_1^2(I_{C_2^{2}})= A_r^3,\\
        \varphi^m(T_2^1)&=
        A_r^3(I_{C_3^{r+2}})\sqcup A_{r-1}^3(I_{C_3^{r+1}})\sqcup \cdots\sqcup A_2^3(I_{C_3^{4}})\sqcup A_1^3(I_{C_3^{3}})= A_r^4,\\
        \varphi^m( T_3^1)&=
        A_r^4(I_{C_4^{r+3}})\sqcup A_{r-1}^4(I_{C_4^{r+2}})\sqcup \cdots\sqcup A_2^4(I_{C_4^{5}}) \sqcup A_1^4(I_{C_4^{4}})= A_r^5,\\
     \vdots   &\quad\quad\quad\quad\quad\quad\quad\quad\quad\quad\vdots\\
        \varphi^m(T_{\frac{m+6}{2}}^1)&=
        A_r^{r-3}(I_{C_{(m-4)/2}^{m-2}})\sqcup A_{r-1}^{r-3}(I_{C_{(m-4)/2}^{m-3}})\sqcup \cdots\sqcup 
         A_1^{r-3}(I_{C_{(m-4)/2}^{(m-2)/2}}) \\
         &=A_r^{r-2}.
 \end{align*}      
Since
 \begin{align*}
  \varphi^m(T^1_{\frac{m+6}{2}} \cup T^2_{\frac{m+6}{2}}) & = A_r^{r-2}(I^{C_{\frac{m-4}{2}}^{m-2}}) \sqcup A_{r-1}^{r-2} \\
  \varphi^m(\cup_{i=1}^3 T^i_{\frac{m+6}{2}} ) & = A_r^{r-2}(I^{C_{\frac{m-4}{2}}^{m-2} + C_{\frac{m-6}{2}}^{m-3}}) 
  \sqcup   A_{r-1}^{r-2}(I^{C_{\frac{m-6}{2}}^{m-3}}) \sqcup A_{r-2}^{r-2} \\
    \varphi^m(\cup_{i=1}^4 T^i_{\frac{m+6}{2}}) & = A_r^{r-2}(I^{C_{\frac{m-4}{2}}^{m-2} + C_{\frac{m-4}{2}}^{m-3} + C_{\frac{m-4}{2}}^{m-4}} )   \sqcup A_{r-1}^{r-2}(I^{C_{\frac{m-4}{2}}^{m-3}+C_{\frac{m-4}{2}}^{m-4}}) \\
& \sqcup A_{r-2}^{r-2}(I^{C_{\frac{m-4}{2}}^{m-4}}) \sqcup A_{r-3}^{r-2} \\
  \vdots \quad \quad \quad \quad& \quad\quad\quad\quad\quad \quad\quad\quad \vdots \\
     \varphi^m(\cup_{i=1}^{\frac{m+4}{2}}T^i _{\frac{m+6}{2}}) & = A_r^{r-2}(I^{C_{\frac{m-4}{2}}^{m-2} + C_{\frac{m-4}{2}}^{m-3} + \cdots + C_{\frac{m-4}{2}}^{\frac{m-4}{2}} })  \sqcup \cdots \sqcup 
  A_{1}^{r-4}(I^{ C_{\frac{m-4}{2}}^{\frac{m-4}{2}}})  \\
   & = A_r^{r-2}(I_{C_{(m-2)/2}^{m-1}}) 
    \sqcup A_{r-1}^{r-2}(I_{C_{(m-2)/2}^{m-2}})\sqcup \cdots\sqcup  A_1^{r-2}(I_{C_{(m-2)/2}^{(m-2)/2}})\\
    &=A_r^{r-1}.
  \end{align*}
Therefore  by lemma \ref{phiinyct} and by  \ref{UnionTriag001} we have
        \begin{align*}      
{\EuScript L}_r & = \varphi^m\big(T_m\big)\\
        &=A_r^{r-1}.
        \end{align*}    
and this concludes the proof.
\end{proof}
\begin{theorem}\label{ultth1}
Let $2\leq k \leq n$ two arbitrary integers and let $r=\lfloor \frac{n+2}{2}\rfloor$ then the matrix \ref{Lmatrix} and the definition matrix \ref{Ass-Matrix} are equal ${\EuScript A}_{n-\lfloor \frac{k-2}{2}\rfloor}^{\lfloor k/2 \rfloor}=A_{n-\lfloor \frac{k-2}{2}\rfloor}^{\lfloor k/2\rfloor}$
\end{theorem}
\begin{proof}
Induction in $k$. So
If $k=n$ then ${\EuScript L}_r=A_{r}^{r-1}$.\\
Now is  $k=n-1$ then ${\EuScript A}_{n-\frac{(n-1)-2}{2}}^{\frac{n-1}{2}}={\EuScript A}_{\frac{(n+1)+2}{2}}^{\frac{(n+1)+2}{2}-2}={\EuScript A}_{\overline{r}}^{\overline{r}-2}$ and is the incidence matrix of the configuration $(C_{\frac{m}{2}}(\Sigma_m), S_{P_{\alpha}})_{P_{\alpha}\in C_{\frac{m-2}{2}}(\Sigma_m)}$ for $m=n+1$ and $\overline{r}=\frac{m+2}{2}$, moreover $\varphi^m(T_m)=A_{\overline{r}}^{\overline{r}-1}$ and $T^1_{\frac{m+6}{2}}\subset T_m$ so we have to ${\EuScript A}_{\overline{r}}^{\overline{r}-2}=\varphi^m(T^1_{\frac{m+6}{2}})=A_{\overline{r}}^{\overline{r}-2}$      
If $k=n-2$ then ${\EuScript A}_{n-\frac{(n-2)-2}{2}}^{\frac{n-2}{2}}={\EuScript A}_{\frac{(n+2)+2}{2}}^{\frac{(n+2)+2}{2}-3}={\EuScript A}_{\overline{r}}^{\overline{r}-3}$ and is the incidence matrix of the configuration $(C_{\frac{m}{2}}(\Sigma_m), S_{P_{\alpha}})_{P_{\alpha}\in C_{\frac{m-2}{2}}(\Sigma_m)}$ for $m=n+2$ and $\overline{r}=\frac{m+2}{2}$,furthermore $\varphi^m(T_m)=A_{\overline{r}}^{\overline{r}-1}$ and $T^1_{\frac{m+4}{2}}\subset T_m$ so ${\EuScript A}_{\overline{r}}^{\overline{r}-3}=\varphi^m(T^1_{\frac{m+6}{2}})=A_{\overline{r}}^{\overline{r}-3}$\\
 And continuing in this way we have to $k=2$ so ${\EuScript A}_{n-\lfloor \frac{2-2}{2}\rfloor}^{\lfloor 2/2 \rfloor}=
 {\EuScript A}_n^1=A_n^1$
\end{proof}
The following corollary follows directly from \ref{ultth1}
\begin{corollary}\label{ultth11}
Let $2\leq k \leq n$ two arbitrary integers then
 $$   {\EuScript A}_{n-\lfloor \frac{k-2}{2}\rfloor}^{\lfloor k/2\rfloor}= \left[ \begin{tabular}{c|c}
$ A_{n-\lfloor \frac{k-2}{2}\rfloor}^{\lfloor k/2\rfloor-1}$ & $ 0 $ \cr\hline $Id$  & $ A_{n-\lfloor \frac{k-2}{2}\rfloor-1}^{\lfloor k/2\rfloor}$
\end{tabular} \right]$$
where $Id$ is the  identity matrix of order $ C_{\lfloor k/2  \rfloor-1}^{n-\lfloor (k-2)/2\rfloor+\lfloor k/2  \rfloor-2}$
 \end{corollary}
\qed\\
From all of the above, the following theorem can be deduced:
\begin{theorem}\label{finalth}(Bizarre version of Bruijn-Erd\"os Theorem)\\
Let $2\leq k \leq n$ be integers, then the set  $C_{\frac{k}{2}}(\Sigma_n)$ of cardinality $N:=C^n_{\frac{k}{2}}$, and the family of subsets $\{S_{P_{\alpha}} \subset C_{\frac{k}{2}}(\Sigma_n) : P_{\alpha}\in C_{\frac{k-2}{2}}(\Sigma_n) \}$  of cardinality  $M=C^n_{\frac{k-2}{2}}$ satisfy
\begin{description}
\item[I)] $|S_{P_{\alpha}}|=n-(k-2)/2$
\item [II)]  $|S_{P_{\alpha}}\cap S_{P_{\overline{\alpha}}} |\leq 1$
\item[III)] Each $S_{P_{\alpha}}$ has nonempty interseccion  with exactly $k/2$ of the other subsets
\end{description}
Moreover exist matrices $B_f$ and  ${\EuScript A}_{n-\frac{k-2}{2}}^{k/2}$ such that
\begin{description}
\item[ i)] $ B_f$ determines all $k$-isotropic subspaces bajo la inclusion de pl\"ucker of a symplectic vector space $E$ of dimension $2n$
\item[ii)] $B_f$ is direct sum of matrices ${\EuScript A}_{n-\frac{k-2}{2}}^{k/2}$ 
\end{description}
where  ${\EuScript A}_{n-\frac{k-2}{2}}^{k/2}$ satisfy \\
\begin{description}
\item[a)]  is a $(0, 1)$-matrix, sparse
\item[b)]  has $n-(k-2)/2$ ones in each row and $k/2$ ones in each
column.
\item[c)] is built with an recursive algorithm, and 
\item [d)] it is fractal matrix
$${\EuScript A}_{n-\frac{k-2}{2}}^{k/2} = \left[ \begin{tabular}{c|c}
 $A_{n-\frac{k-2}{2}}^{k/2-1}$ & $ 0 $ \cr\hline $Id$  & $ A_{n-\frac{k-2}{2}-1}^{k/2}$
\end{tabular} \right]$$
\end{description}
where $A_{n-\frac{k-2}{2}}^{k/2-1}$ and $A_{n-\frac{k-2}{2}-1}^{k/2}$ they are built with the same recursive algorithm and satisfy properties similar to $a)$ and $b)$ above. 
\end{theorem}
\begin{proof}
Clearly $N>M$, now $I)$, $II)$, and $III)$ are followed by lemma \ref{onesroow}, corollary \ref{canes1} and lemma \ref{unoscolumn}. For $i)$ and $ii)$ see theorems \ref{th211m} $a)$  and  \ref{sumdirect} respectively. So for 
$a)$ and $b)$ see proposition \ref{princmatrx}.  Of theorem \ref{ultth1} se sigue $c)$ and $d)$.
\end{proof}
\section{An method for to obtain the ${\mathbb F}_q$-rational points of $IG(k, E)$}
Let ${\mathbb F}_q$ be a finite field with $q$ elements, and denote
by $\overline{\mathbb F}_q$ an algebraic closure of ${\mathbb F}_q$. For a vector space $E$  over ${\mathbb F}_q$ of finite dimension $k$, let $\overline{E}=E\otimes_{{\mathbb F}_q}\overline{\mathbb F}_q$ be the corresponding vector space over the algebraically closed field $\overline{\mathbb F}_q$. We will be considering algebraic varieties in the projective space ${\mathbb P}(\overline{E})={\mathbb P}^{k-1}(\overline{\mathbb
F}_q)$. Recall that a projective variety $X\subseteq{\mathbb P}^{k-1}(\overline{\mathbb
F}_q)$ is defined over the finite field ${\mathbb F}_q$ if its
vanishing ideal can be generated by polynomials with coefficients in
${\mathbb F}_q$. 
If the field of definition of the simplectic vector space is ${\mathbb F}_q$ then to the set
\begin{equation}\label{rationpoints1}
IG(k, \overline{E})({\mathbb F}_q)= Z\langle Q_{\alpha^{\prime}, \beta^{\prime}},   \Pi_{\alpha_{rs}}, x_{\alpha}^q-x_{\alpha}\rangle
\end{equation}
To the set \ref{rationpoints1} we call ${\mathbb F}_q$-rational points of $IG(k, \overline{E})$.\\
The cardinality  of ${\mathbb F}_q$-rational points is
\begin{equation}\label{rationpoints2}
|IG(k, \overline{E})({\mathbb F}_q)|=\prod_{i=0}^{k-1}\frac{q^{2n-2i}-1}{q^{i+1}-1}
\end{equation}
The  problem of obtaining its ${\mathbb F}_q$-rational points 
 requires an algorithm, whose complexity computacional 
 depends directly on the  properties of  the matrix  $B_f$. 
Solve the system $B_fw^T=0$ using Gaussian elimination requires at the most $O((C^{2n}_{n})^3)$ arithmetic operations, where
$n=C^{2n}_{n}$. But by the theorem \ref{sumdirect} $B_f$ direct addition of $(0, 1)$-matrices \ref{Lmatrix}  which are 
regular, sparce and structured, which allows the use of iterative methods, which are equally or better efficient algorithms than Gaussian elimination.  So then we can find an efficient method to solve the system
${\EuScript A}_{n-\frac{k-2}{2}}^{k/2}x=0$ where they are required $O((C_{\frac{k}{2}-1}^n)^3)$ arithmetic operations.
So here we propose a "method" to find all rational points:
\begin{equation}\label{Algoritm2}
 Method\; A
\end{equation}
\begin{itemize}
\item[{\rm step 1}] Input: $w=(x_{\alpha})_{\alpha\in I(k, m)}\in {\mathbb F}_q^{C^{2n}_{k}}$ such that $B_fw^T=0$.
\item[{\rm step 2}] If $Q_{\alpha^{\prime},\beta^{\prime}}(w)= 0$, for all $\alpha^{\prime}\in I(k-1,m)$,  $\beta^{\prime} \in I(k+1,m)$, then $w\in IG(k, \overline{E})({\mathbb{F}}_q)$ is stored in a set called $\Sigma_1$ and go to step 1, now with $w\in \ker B_f-\Sigma_1$.
\item[{\rm step 3}]  If $Q_{\alpha^{\prime},\beta^{\prime}}(w)\neq 0$  for some $\alpha^{\prime}\in I(k -1,m)$,  $\beta^{\prime} \in I(k+1,m)$, then $w\notin IG(k, \overline{E})({\mathbb{F}}_q)$ and $w$ is stored in a set called $\Sigma_2$, and go to step 1 but now with $w\in \ker B_f-(\Sigma_1 \cup \Sigma_2)$.
\item[{\rm step 4}] The process ends when the cardinality of $\Sigma_1$ in step 2 is equal to $|IG(k, \overline{E})({\mathbb F}_q)|$  or the cardinality of $\Sigma_2$ in step 3 is equal to $q^b-|IG(k, \overline{E})({\mathbb F}_p)|$ where $b:=C^{2n}_{k}$.
\item[{\rm step 5}] Output: Depending of which one of the following occurs first:
\begin{enumerate}
\item[(1)] $\Sigma_1$ if its cardinality   equals $|IG(k,\overline{E})({\mathbb F}_q)|$, or 
\item[(2)] $\Sigma_2$ if its cardinality   equals  $q^b-|IG(k, \overline{E})({\mathbb F}_q)|$ where $b=C^{2n}_{k}$.
\end{enumerate}
\end{itemize}

The following proposition  justifies the previous algorithm.

\begin{proposition}  
With the previous notation, 
if  $|\Sigma_1|=|IG(k, E)({\mathbb F}_q)|$ then $IG(k, \overline{E})({\mathbb F}_q)=\Sigma_1$ or if $|\Sigma_2|=q^b-|IG(k, \overline{E})({\mathbb F}_q)| $ then $IG(k, \overline{E})({\mathbb F}_q)={\mathbb F}_q^b-\Sigma_2$,  where $b=\binom{m}{k}$.
\end{proposition}

\begin{proof}
First note that ${\mathbb F}_q^b=\Sigma_1\cup \Sigma_2$ and $\Sigma_1\cap \Sigma_2=\emptyset$. Now, 
if $w=(x_{\alpha})_{\alpha\in I( k, m)}\in {\mathbb P}(\wedge^{k}\overline{E})$ such that $Q_{\alpha^{\prime}, \beta^{\prime}}(w)=0$  for all  $\alpha^{\prime}\in I(\ell-1,2n)$ and $\beta^{\prime}\in I(\ell+1,m)$ and  $B_fw^T=0$, then 
$w\in X({\mathbb F}_q)$.  Therefore,
$\Sigma_1 \subseteq IG(k, \overline{E})({\mathbb{F}}_q)$. As $|\Sigma_1| =
IG(k, \overline{E})({\mathbb F}_q)|$, we have $\Sigma_1 =IG(k,\overline{E})({\mathbb F}_q)$.
Clearly ${\mathbb F}_q^b-\Sigma_2\subseteq \Sigma_1$ and $|{\mathbb F}_q^b-\Sigma_2|=q^b-|\Sigma_2|=|\Sigma_1|$.
\end{proof}




\end{document}